\documentclass[12pt,a4paper]{article}
\newif\ifdraft 
\draftfalse    

\usepackage{amssymb}
\usepackage{amsmath}
\usepackage{amsthm}
\usepackage{thmtools}
\usepackage{enumitem}
\usepackage{cite}
\usepackage[all]{xy}
\usepackage{hyperref}
\usepackage{bm}
\usepackage{parskip}
\usepackage[a4paper,margin=2cm]{geometry}
\usepackage[utf8]{inputenc}
\usepackage{tikz}
\usepackage{ifthen}
\usepackage{multicol}

\usetikzlibrary{arrows.meta}

\begingroup
    \makeatletter
    \@for\theoremstyle:=definition,remark,plain\do{%
        \expandafter\g@addto@macro\csname th@\theoremstyle\endcsname{%
            \addtolength\thm@preskip\parskip
            }%
        }
\endgroup

\makeatletter
\xydef@\txt@ii#1{\vbox{\vspace*{-5pt}%
 \let\\=\cr
 \tabskip=\z@skip \halign{\relax\hfil\txtline@@{##}\hfil\cr\leavevmode#1\crcr}}}
\makeatother

\theoremstyle{definition}
\newtheorem{thm}{Theorem}[section]
\newtheorem{lem}[thm]{Lemma}
\newtheorem{cor}[thm]{Corollary}
\newtheorem{defn}[thm]{Definition}
\newtheorem{propn}[thm]{Proposition}
\newtheorem*{thm*}{Theorem}
\newtheorem{notn}[thm]{Notation}
\newtheorem{qn}[thm]{Question}
\newtheorem{setup}[thm]{Setup}

\newtheorem{props}[thm]{Properties}

\theoremstyle{remark}
\newtheorem{rk}[thm]{Remark}
\newtheorem*{rk*}{Remark}
\newtheorem{rks}[thm]{Remarks}
\newtheorem*{rks*}{Remarks}
\newtheorem{ex}[thm]{Example}
\newtheorem*{ex*}{Example}
\newtheorem{exs}[thm]{Examples}

\newtheoremstyle{custthm}{\parskip}{}{\normalfont}{}{\bfseries}{.}{ }{\thmname{#1} \thmnote{#3}}
\theoremstyle{custthm}

\newcommand{\SPS}{\mathbf{SPS}}

\renewcommand{\div}{\mathrm{div}}
\newcommand{\supp}{\mathrm{supp}}
\newcommand{\LT}{\mathrm{LT}}
\newcommand{\LM}{\mathrm{LM}}
\newcommand{\LC}{\mathrm{LC}}

\newcommand\blfootnote[1]{%
  \begingroup
  \renewcommand\thefootnote{}\footnote{#1}%
  \addtocounter{footnote}{-1}%
  \endgroup
}

\begin{document}

\numberwithin{equation}{section}
\binoppenalty=\maxdimen
\relpenalty=\maxdimen

\title{Monomial methods in iterated local skew power series rings}
\author{Billy Woods}
\date{\today}
\maketitle
\begin{abstract}
Let $A = \mathbb{F}_p$ or $\mathbb{Z}_p$, and let $R = A[[x_1]][[x_2; \sigma_2, \delta_2]]\dots[[x_n;\sigma_n,\delta_n]]$, an iterated local skew power series ring over $A$. Under mild conditions, we show that (multiplicative) monomial orders exist, and develop the theory of Gr\"obner bases for $R$. We show that all rank-2 local skew power series rings over $\mathbb{F}_p$ satisfy polynormality, and give an example of a rank-2 local skew power series ring over $\mathbb{Z}_p$ which is a unique factorisation domain in the sense of Chatters-Jordan.

\end{abstract}

\tableofcontents

\section{Introduction}
 The main objects of study in this paper are certain kinds of iterated \emph{skew power series} rings: that is, given a base ring $A$ and skew derivations $(\sigma_i, \delta_i)$ satisfying certain conditions (see \S \ref{subsec: SPSRs} for definitions and details), we intend to study the topological ring
\begin{align}\label{eqn: SPSR}
R = A[[x_1; \sigma_1, \delta_1]][[x_2; \sigma_2, \delta_2]] \dots [[x_n; \sigma_n, \delta_n]].
\end{align}
Basic facts and questions about skew power series rings, with many references to the existing literature, are reviewed in \S \ref{subsec: SPSRs}.

 Much of the motivation for this paper stems from the theory of \emph{Iwasawa algebras}, whose definitions and properties are summarised in \S \ref{subsec: iwasawa algebras}. Briefly: these are completed group algebras $\mathbb{F}_p[[G]]$ or $\mathbb{Z}_p[[G]]$ of certain virtually pro-$p$ groups $G$, and arise very often in modern number theory, but are still not very well understood. We will see below that, for certain large classes of $G$, these Iwasawa algebras can be viewed as iterated skew power series rings in the sense appropriate to this paper.

 Unlike skew polynomial rings $R[x; \sigma, \delta]$, which are reasonably well-understood, skew power series rings $R[[x; \sigma, \delta]]$ remain mysterious. While skew power series rings are trickier to define, they have been known and studied in a sufficiently general context for our purposes since the beginning of the century \cite[\S 2]{venjakob}, \cite[\S 1]{SchVen06}. However, they largely remain computationally intractable. We hope to begin to rectify this situation with the current paper.

\subsection{Monomial orders}

 We begin by assuming that our base ring $A$ is \emph{either}
\begin{itemize}
\item a finite field $k$ of characteristic $p$, \emph{or}
\item a commutative complete discrete valuation ring whose residue field $k$ is a finite field of characteristic $p$.
\end{itemize}

 In the first case, and with $R$ as in (\ref{eqn: SPSR}), it is known that elements of $R$ have a unique representation with respect to the ordered tuple of variables $(x_1, \dots, x_n)$. Explicitly, writing $\bm{x}^\alpha := x_1^{\alpha_1} \dots x_n^{\alpha_n}$ for $\alpha = (\alpha_1, \dots, \alpha_n)\in \mathbb{N}^n$, we have
\begin{align}\label{eqn: R as a set}
R = \left\{ \sum_{\alpha\in \mathbb{N}^n} r_\alpha \bm{x}^\alpha \;\middle|\; r_\alpha\in k\right\}.
\end{align}
 Given appropriate conditions on the skew derivations $(\sigma_i, \delta_i)$, this $R$ will admit a \emph{monomial order} with respect to this tuple of variables: that is, a well-ordering $\preceq$ on $\mathbb{N}^n$, assigning to each nonzero $r = \sum_{\alpha\in \mathbb{N}^n} r_\alpha \bm{x}^\alpha \in R$ its \emph{least monomial} $\LM(r) = \min_{\preceq} \{\alpha\in\mathbb{N}^n : r_\alpha \neq 0\}$, such that this assignment respects the addition and multiplication in $R$ (Definition \ref{defn: monomial order}).

 Somewhat surprisingly, we are able to deal with the second case near-simultaneously by treating a uniformiser $\pi$ for $A$ as one of the variables. In this case, elements $r\in R$ still admit representations of a similar form: $r = \sum_\alpha r_\alpha \pi^{\alpha_0} x_1^{\alpha_1} \dots x_n^{\alpha_n}$, where now $\alpha = (\alpha_0, \dots, \alpha_n) \in\mathbb{N}^{n+1}$, and this representation is unique for an appropriate choice of restriction on the possible coefficients $r_\alpha$. (The coefficients $r_\alpha$ no longer necessarily behave nicely under addition and subtraction, but this does not prevent us from working with monomial orders: this somewhat awkward setup is perhaps most naturally viewed as extending the discrete $\pi$-adic valuation $A \to \mathbb{N}\cup\{\infty\}$ to a rank-$(n+1)$ valuation $\LM: R\to \mathbb{N}^{n+1}\cup\{\infty\}$.)

\subsection{Conditions on the skew power series extensions}

 Let $A$ satisfy the hypotheses of the previous subsection. In order for the iterated skew power series extension (\ref{eqn: SPSR}) to be defined, certain conditions must be imposed on the $(\sigma_i, \delta_i)$.

 If $B$ is a complete local ring with maximal ideal $\mathfrak{m}$, then the skew derivation $(\sigma, \delta)$ on $B$ is \emph{local} if $\sigma(\mathfrak{m}^j) = \mathfrak{m}^j$ and $\delta(\mathfrak{m}^j) \subseteq \mathfrak{m}^{j+1}$ for all $j\in\mathbb{N}$: this implies that $B[[x; \sigma, \delta]]$ exists and is again a complete local ring. We will assume throughout this paper that all $(\sigma_i, \delta_i)$ are local, so that $R$ is again complete local. This is the primary case of interest to Iwasawa algebras.

 \textit{Remark.} In fact, when dealing with Iwasawa algebras, it is usually possible to assume without much difficulty that $\delta(\mathfrak{m}^j) \subseteq \mathfrak{m}^{j+2}$. This is enough to imply that monomial orders exist: for instance, any additive total order refining the $\mathfrak{m}$-adic filtration, such as the degree lexicographic order, will be a monomial order. However, in practice, this order does not appear to be very useful: in our applications below, we will always use the \emph{lexicographic} order, suggesting that this order is likely to be of more use in general.

In order for the lexicographic order to be compatible with the multiplication in $R$, we will need to impose further restrictions on the $(\sigma_i, \delta_i)$. Specifically, we focus on the subclass of iterated local skew power series rings studied in the paper \cite{woods-SPS-dim}, namely those that are \emph{triangular}. Precise definitions are given in \S \ref{subsec: SPSRs} below.

This additional condition may seem restrictive, but it turns out \cite[\S 2.4]{woods-SPS-dim} that many natural and interesting examples of skew power series rings are indeed triangular, including $q$-commutative power series rings, nilpotent and supersoluble Iwasawa algebras, and certain completed quantum groups, so we hope that the current work will find broad application.

\subsection{Right-division, Gr\"obner bases and Buchberger's criterion}

 In all of the following theorems, we assume that \emph{either}
\begin{itemize}
\item $R = k[[x_1; \sigma_1, \delta_1]][[x_2; \sigma_2, \delta_2]] \dots [[x_n; \sigma_n, \delta_n]]$ ($k = \mathbb{F}_{p^m}$), \emph{or}
\item $R = A[[x_2; \sigma_2, \delta_2]][[x_3; \sigma_3, \delta_3]] \dots [[x_n; \sigma_n, \delta_n]]$ ($A$ a commutative complete discrete valuation ring with residue field $\mathbb{F}_{p^m}$, and we set $x_1$ to be a fixed uniformiser of $A$ -- note carefully the renumbered indices!),
\end{itemize}
where all skew derivations are local.

Our main results are as follows.

\begin{itemize}
\item (Existence of monomial orders.) If all skew derivations are \emph{triangular}, then the lexicographic order, the degree lexicographic order and the degree reverse lexicographic order are all monomial orders on $R$ (\textbf{Theorem \ref{thm: LM compatibility}}).
\item (Multivariate right-division.) 
Fix a monomial order $\preceq$ on $\mathbb{N}^n$, and let $F = (f_1, \dots, f_s)$ be an ordered tuple of elements of $R$. Then every $f\in R$ can be written as $f = q_1 f_1 + \dots + q_s f_s + r$, where the \emph{quotients} $q_1, \dots, q_s$ and the \emph{remainder} $r =: \overline{f}^F$ are unique subject to appropriate conditions on their support (\textbf{Theorem \ref{thm: division theorem}}).
\item (Gr\"obner bases.) Fix a monomial order $\preceq$ on $\mathbb{N}^n$, and let $\LM = \LM_\preceq: R\setminus \{0\} \to \mathbb{N}^n$ be the function sending nonzero $r\in R$ to the $\preceq$-\emph{least monomial} in its support. We define \emph{Gr\"obner bases} for nonzero left ideals $I\lhd R$ (\textbf{Definition \ref{defn: Groebner basis}}). If $G$ is a Gr\"obner basis for $I$, and $f\in R$ is arbitrary, then $f\in I$ if and only if its remainder on right-division by $G$ is zero (\textbf{Corollary \ref{cor: left ideal membership}}). Moreover, if $I\subseteq J$ are left ideals of $R$, and $\LM(I) = \LM(J)$, then $I = J$ (\textbf{Corollary \ref{cor: equality of nested ideals}}).
\item (Buchberger's criterion.) Suppose $I$ is a nonzero left ideal of $R$ and $G = (g_1, \dots, g_s)$ is a left ideal basis for $I$. Then $G$ is a Gr\"obner basis for $I$ if and only if each S-element $S(g_i,g_j)$ (\textbf{Definition \ref{defn: S-elements}}) has zero remainder on right-division by $G$ (\textbf{Theorem \ref{thm: S-elements determine Groebner bases}}).
\end{itemize}

We also show that the associated graded ring of $R$ with respect to $\LM$ is a $q$-commutative polynomial ring (\textbf{Corollary \ref{cor: associated graded is q-commutative}}), and prove a multivariate form of the Weierstrass preparation theorem (\textbf{Corollary \ref{cor: weierstrass}}), but we do not use these minor results in any serious way.

\subsection{Applications}

We give two explicit applications to skew power series rings.

The below example is simply a direct calculation using the lexicographic order $\preceq$, and does not rely on the multiplicativity of $\preceq$ or on the theory of Gr\"obner bases. However, we hope that it will be possible to extend these calculations to rings with more complicated structure using the subsequent results.

\textit{Example} \ref{ex: 3d over Zp, yx = xy + p^2}. The unique prime ideal of height 1 of the ring $R = \mathbb{Z}_p[[x, y \;|\; yx - xy = p^2]]$ is $pR$. In particular, $R$ is a noncommutative unique factorisation domain in the sense of \cite{chatters-jordan}.

Recall that a ring $R$ is called \emph{polynormal} if, whenever $J\subsetneq I$ are ideals of $R$, the ideal $I/J$ of $R/J$ contains a nonzero normal element. This property implies that $R$ is \emph{AR-separated}, and in particular satisfies the \emph{strong second layer condition}, a crucial property in the theory of localisation: see \cite[\S 8.1]{jategaonkar} for details.

\textbf{Corollary \ref{cor: 2d rings are polynormal}.} Let $(\sigma, \delta)$ be an arbitrary $\mathbb{F}_p$-linear, local skew derivation on $\mathbb{F}_p[[x]]$. Then $R = \mathbb{F}_p[[x]][[y; \sigma, \delta]]$ is polynormal.

We hope that this result can be extended to more general rings $R$: see Question \ref{qn}. In the special case where this $R = \mathbb{F}_p G$ is the Iwasawa algebra of the soluble group $\mathbb{Z}_p\rtimes \mathbb{Z}_p$ (see \cite[Example 2.2]{venjakob}), this result can be compared to \cite[Theorem 7.1]{venjakob}. See also \cite{jones-abelian-by-procyclic} for important structural results about the Iwasawa algebras of more general abelian-by-procyclic groups $G$.

\section{Preliminaries}

\subsection{Complete discrete valuation rings}\label{subsec: CDVRs}

 Throughout this paper, discrete valuation rings $A$ are always \emph{commutative}, and to avoid edge cases we will also assume that $J(A)\neq 0$, i.e. that they are not fields. We assume familiarity with the basic properties and terminology of discrete valuation rings, such as can be found at the beginning of \cite[Chapter I, \S 1]{Ser79}. 

Let $A$ be a complete discrete valuation ring with residue field $k = A/J(A)$. The natural quotient map $q: A\to k$ is of course surjective, and we will write $\iota: k\to A$ for an arbitrary map of sets which splits $q$, i.e. such that $q\circ \iota = \mathrm{id}_k$.

 \begin{lem}\label{lem: standard form in A}
Fix a uniformiser $x\in J(A)$. Then every element $a\in A$ can be written as $a = \sum_{j=0}^\infty a_j x^j$ for some choice of coefficients $a_j \in \iota(k)$, and moreover this representation is unique.
\end{lem}

\begin{proof}
Given $a\in A$, set $a_0 = \iota(q(a))$: then $a \equiv a_0$ mod $xA$. Then proceed inductively: if $a \equiv a_0 + a_1 x + \dots + a_{j-1} x^{j-1}$ mod $x^j A$ for any $j\geq 1$, write $a - (a_0 + a_1 x + \dots + a_{j-1} x^{j-1}) = a' x^j$ for some $a'\in A$, and choose $a_j = \iota(q(a'))$. Then $a \equiv a_0 + a_1 x + \dots + a_{j-1} x^{j-1} + a_j x^j$ mod $x^{j+1} A$, and as $A$ is complete, this sequence of partial sums $(\sum_{\ell=0}^j a_\ell x^\ell)_{j\geq 0}$ converges to $a$. This expression is unique as the choices of $a_j$ were unique at each stage.
\end{proof}

 \begin{defn}\label{defn: standard form in A}
We will say that the expression $a = \sum_{j=0}^\infty a_j x^j\in A$ is in \emph{standard form} (with respect to a fixed $\iota$ and $x$) if the coefficients $a_j$ lie in $\iota(k)$.
\end{defn}

In general, there is no hope for this standard form to be well-behaved with respect to addition, even with judicious choice of $\iota$ and $x$. However, it does have the following properties:

 \begin{props}\label{props: standard form}
Suppose that $a, b, c$ are written in standard form as above. Write $v$ for the (normalised) $x$-adic valuation function $A\to\mathbb{N}\cup\{\infty\}$.

\begin{enumerate}[label=(\roman*)]
\item $v(a) = m$ if and only if $a_m \neq 0$ and $a_i = 0$ for all $i < m$.
\item Suppose $v(a) = m$, $v(b) = n$, and $ab = c$. Then $v(c) = m+n$.
\item Suppose $v(a) = m$, $v(b) = n$, and $ab = c$. Then $a_m b_n \equiv c_{mn} \bmod J(A)$.
\item Suppose $v(a) = m \leq v(b) = n$, and $a+b=c$. Then $a_m + b_m \equiv c_m \bmod J(A)$.
\end{enumerate}
\end{props}

\begin{rk*}
The beginnings of the later theory are already visible here. If $0\neq a\in A$ is written in standard form as $a = \sum_{j=0}^\infty a_j x^j$, its \emph{least monomial} is $\LM(a) := v(a)$, and -- crucially -- this does not depend on our choice of $\iota$ or $x$. If $\LM(a) = n$, then the \emph{least coefficient} of $a$ is $\LC(a) = a_n$, and modulo $J(A)$ this also does not depend on $\iota$ or $x$. Properties (ii--iv) are properties that we will later expect of monomial orders in general.

We will also later assume that $\iota$ is multiplicative, at which point the congruence in part (iii) will become an equality inside $\iota(k)$.
\end{rk*}

\subsection{Skew power series rings}\label{subsec: SPSRs}

 A \emph{skew derivation} on a ring $R$ is a pair $(\sigma, \delta)$, where $\sigma$ is an automorphism of $R$ and $\delta$ is a \emph{(left) $\sigma$-derivation}, i.e. a linear endomorphism of $R$ such that $\delta(rs) = \delta(r) s + \sigma(r)\delta(s)$ for all $r,s\in R$. (Some authors allow more general endomorphisms $\sigma$, but we do not need this level of generality.)

 Any nontrivial skew derivation $(\sigma, \delta)$ on $R$ induces a (noncommutative) multiplication on the polynomial $R$-module $R[x]$ given by
\begin{align}\label{eqn: skew derivation multiplication formula}
xr = \sigma(r)x + \delta(r)
\end{align}
for all $r\in R$. The resulting ring is called a \emph{skew polynomial ring}, and is usually written $R[x;\sigma,\delta]$. These are now classical and reasonably well-understood objects, discussed in detail in the literature, e.g. \cite{MR,GooLet94,brown-goodearl}. Important examples of iterated skew polynomial rings include Weyl algebras, group algebras of poly-(infinite cyclic) groups, enveloping algebras of soluble Lie algebras, and many quantum groups.

 If the multiplication (\ref{eqn: skew derivation multiplication formula}) also extends to a continuous multiplication on the (formal) power series $R$-module $R[[x]]$, the resulting (topological) ring is then called a skew power series ring, and is denoted $R[[x; \sigma, \delta]]$.

 In the case where $\delta = 0$, skew power series rings (usually written $R[[x; \sigma]]$) are classical objects: \cite[Remark 1.9.5]{MR} dates them back to the late 19th century. While some of their basic ring-theoretic properties are now well-known \cite[\S 2.3]{Coh95}, \cite[\S 1.4, Theorem 7.5.3(v)]{MR}, much is still under development: among the many disparate papers dealing with these rings, we mention only \cite[\S 5]{SchVen06}, \cite[\S 2]{schneider-venjakob-localisations}, \cite{letzter-wang-goldie}.

 If $\delta \neq 0$, it is often \emph{not} the case that (\ref{eqn: skew derivation multiplication formula}) extends to a continuous multiplication on $R[[x]]$, i.e. the ring $R[[x; \sigma, \delta]]$ often does not exist. Several authors have studied more general cases: we mention Bergen and Grzeszczuk in particular, who have worked on the case where $\delta$ acts locally nilpotently on $R$, e.g. \cite{bergen-grzeszczuk-skew}. Unfortunately, in the primary cases of interest to us, even this is not enough: for us, $\delta$ will often only be \emph{topologically} nilpotent in an appropriate sense.

\begin{defn}
Let $(R, \mathfrak{m})$ be a complete local ring. Then the skew derivation $(\sigma, \delta)$ on $R$ is \emph{local} if $\sigma(\mathfrak{m}^i) = \mathfrak{m}^i$ and $\delta(\mathfrak{m}^i) \subseteq \mathfrak{m}^{i+1}$ for all $i\geq 0$.
\end{defn}

 If $R$ is a complete local ring and $(\sigma,\delta)$ is a local skew derivation on $R$, then the skew power series ring $R[[x; \sigma, \delta]]$ is known to exist \cite[\S 1]{SchVen06}, \cite[\S 2]{venjakob}. In this paper, we will restrict ourselves to such \emph{local} skew power series rings, and will not consider the question of existence any further. Many of the basic properties of local skew power series rings (or, more generally, \emph{filtered} skew power series rings) are proved in \cite{letzter-noeth-skew,schneider-venjakob-localisations,woods-SPS-dim,jones-woods-2}.  We single out one important fact: in this context, $R[[x; \sigma, \delta]]$ is again a complete local ring, so this construction can be iterated as follows.

\begin{props}\label{props: iterated SPSRs}
Let $R_0 = A$ be a complete local ring. Then, for $1\leq i\leq n$, suppose that $(\sigma_i, \delta_i)$ is a local skew derivation on $R_{i-1}$, and form the (complete local) ring $R_i = R_{i-1}[[x_i; \sigma_i, \delta_i]]$. Finally, set $R = R_n$, so that we have
$$R = A[[x_1; \sigma_1, \delta_1]] \dots [[x_n; \sigma_n, \delta_n]].$$
This is an iterated form of the construction studied in \cite[\S 3]{letzter-noeth-skew}; some special cases were studied in \cite{wang-quantum}. In the notation of \cite[\S 1]{woods-SPS-dim}, this was denoted $R\in \SPS^n(A)$. We record some basic properties of the ring $R$.

\begin{enumerate}[label=(\roman*)]
\item Every element $r\in R$ can be written as
\begin{align}\label{eqn: expression for r over A}
\displaystyle r = \sum_{\beta\in \mathbb{N}^n} a_\beta \bm{x}^\beta
\end{align}
for a unique choice of coefficients $a_\beta\in A$. (Here and elsewhere in the paper, if $\beta$ is the tuple $(\beta_1, \dots, \beta_n) \in \mathbb{N}^n$, we write as shorthand $\bm{x}^\beta := x_1^{\beta_1} \dots x_n^{\beta_n}$.) Conversely, given any family $(a_\beta)_{\beta\in\mathbb{N}^n}$ of elements of $A$, the sum $\sum_{\beta\in\mathbb{N}^n} a_\beta \bm{x}^\beta$ converges to a well-defined element of $R$.
\item The maximal ideal of $R$ is generated by the maximal ideal of $A$ and the elements $x_1, \dots, x_n$ \cite[\S 2]{venjakob}, \cite[Lemma 1.5]{woods-SPS-dim}.
\item Let $\mathfrak{m}$ be the maximal ideal of $R$. Then $\bigcap_{n\geq 0} \mathfrak{m}^n = 0$ \cite[\S 1]{SchVen06}, \cite[\S 2]{venjakob}, \cite[Proposition 3.7]{letzter-noeth-skew}.
\end{enumerate}
\end{props}

 \begin{defn}
$ $

\begin{enumerate}[label=(\roman*)]
\item Suppose that $A$ is a subring of $R$, and $(\sigma, \delta)$ is a skew derivation on $R$. We will say that $(\sigma, \delta)$ is \emph{$A$-linear} if $\sigma(a) = a$ and $\delta(a) = 0$ for all $a\in A$.
\item Suppose that $R_0 \subseteq \dots \subseteq R_i$ is a sequence of local rings, and $(\sigma, \delta)$ is a local skew derivation on $R_i$. We will say that $(\sigma, \delta)$ is \emph{triangular} (with respect to this sequence) if it restricts to a skew derivation of $R_j$ for each $0\leq j\leq i$: that is, if for each $0\leq j\leq i$, we have $\sigma(R_j) \subseteq R_j$ and $\delta(R_j) \subseteq R_j$.
\end{enumerate}
\end{defn}

As mentioned in the Introduction, the definitions in this subsection are broad enough to include many natural and interesting examples of iterated skew power series rings: see \cite[\S\S 2.4--2.5]{woods-SPS-dim} for examples of local, $A$-linear iterated skew power series rings and how they relate to the triangularity condition.

\subsection{Iwasawa algebras}\label{subsec: iwasawa algebras}

Fix a prime number $p$. A \emph{compact $p$-adic analytic group} may be most simply described as a closed subgroup of the topological group $GL_n(\mathbb{Z}_p)$ for some $n$ \cite[Definition 8.14]{DDMS}.

Much of what is known about the completed group rings $\mathbb{F}_p[[G]]$ and $\mathbb{Z}_p[[G]]$ of such groups (Iwasawa algebras) is presented in the summary paper \cite{ardakovbrown}. We mention \cite{ardakovInv,jones-abelian-by-procyclic,jones-primitive-ideals,woods-catenary,jones-woods-1,letzter-noeth-skew} as examples of more recent papers containing important structural results for subclasses of soluble Iwasawa algebras.

When studying compact $p$-adic analytic groups, it is often advantageous to begin with a large subclass of such groups that are relatively well-behaved, such as the \emph{uniform} groups \cite[Definition 4.1]{DDMS}. If $G$ is a uniform group, then in particular, $G$ admits an \emph{ordered basis}, i.e. there exists an ordered tuple $(g_1, \dots, g_n)$ of elements of $G$ such that the map
\begin{align*}
\mathbb{Z}_p^n &\to G\\
\alpha = (\alpha_1, \dots, \alpha_n) &\mapsto g_1^{\alpha_1} \dots g_n^{\alpha_n} =: \bm{g}^\alpha
\end{align*}
is a homeomorphism \cite[Theorem 4.9]{DDMS}. For each $g_i\in G$, set $x_i = g_i - 1 \in A[[G]]$: then in cases of interest, including the cases where $A = \mathbb{Z}_p$ or $\mathbb{F}_p$, we can write
$$A[[G]] = \left\{ \sum_{\alpha\in \mathbb{N}^d} c_\alpha \bm{x}^\alpha \;\middle|\; c_\alpha\in A\right\},$$
where again if $\alpha = (\alpha_1, \dots, \alpha_n)$ then $\bm{x}^\alpha$ denotes the ordered product $x_1^{\alpha_1} \dots x_n^{\alpha_n}$ \cite[Theorems 7.20 and 7.23(i)]{DDMS}. That is, as a topological $A$-module, $A[[G]]$ is simply a power series module: $A[[G]] \cong A[[x_1, \dots, x_n]]$. Unfortunately, in general, this isomorphism is far too complicated to allow direct computation with the resulting power series.

However, under these assumptions, if $G$ is a soluble group, the ordered basis $(g_1, \dots, g_n)$ can be chosen so that
$$A[[G]] \cong A[[x_1]][[x_2; \sigma_2, \delta_2]] \dots [[x_n; \sigma_n, \delta_n]],$$
a representation of $A[[G]]$ as an iterated local $A$-linear skew power series ring \cite[Example 2.3]{venjakob}. If $G$ is in fact \emph{supersoluble} (e.g. nilpotent), the above representation can also be made triangular \cite[Examples 2.15--2.16]{woods-SPS-dim}.

\subsection{Additive total orderings}\label{subsec: additive total orderings}

Later, we will be dealing with noncommutative analogues of rings such as $R = k[[x_1, \dots, x_n]]$ (for $k$ a commutative ring). Taking this as our blueprint for now, we assume that we are given a ring $R$ and a sequence of $n$ fixed elements $x_1, \dots, x_n\in R$.

\begin{defn}
Let $\alpha = (\alpha_1, \dots, \alpha_n)\in\mathbb{N}^n$. Then $\bm{x}^\alpha$ denotes the \emph{(standard) monomial} $x_1^{\alpha_1} \dots x_n^{\alpha_n}\in R$, and we will refer to $\alpha$ as the \emph{exponent} of $\bm{x}^\alpha$.
\end{defn}

\begin{rk}\label{rk: standard monomials}
Here, the adjective ``standard" denotes that the monomial can be written as $x_{i_1} x_{i_2} \dots x_{i_m}$ with $i_1 \leq i_2 \leq \dots \leq i_m$. We will need to briefly deal with \emph{generalised} monomials in \S \ref{subsec: generalised monomials} below, but we will otherwise use ``monomial" to mean ``standard monomial" by default.
\end{rk}

Assume now that every element $r\in R$ can be written uniquely as
$$r = \sum_{\alpha\in\mathbb{N}^n} r_\alpha \bm{x}^\alpha$$
for some choice of coefficients $r_\alpha\in C$, where $C$ is a fixed central subset of $R$. (As we might expect from our blueprint, this implies that the map $\mathbb{N}^n\to R$ defined by $\alpha \mapsto \bm{x}^\alpha$ is injective.) We set some notation and conventions.

\begin{defn}\label{defn: div, deg, support, etc}
$ $

\begin{enumerate}
\item 
Let $\alpha,\beta\in\mathbb{N}^n$. We will write $\alpha \leq_\div \beta$ for the \emph{divisibility} partial ordering
$$\alpha \leq_\div \beta \Leftrightarrow \beta - \alpha \in \mathbb{N}^n.$$
Note that this terminology is borrowed from the commutative theory: even if $\alpha\leq_\div \beta$ in $\mathbb{N}^n$, it usually does \emph{not} follow that $\bm{x}^\beta$ is divisible by $\bm{x}^\alpha$ as elements of $R$ if the multiplication of $R$ is noncommutative.
\item
Let $\alpha = (\alpha_1, \dots, \alpha_n)\in\mathbb{N}^n$. Then the \emph{total degree} of $\bm{x}^\alpha$ is defined to be $|\alpha| := \alpha_1 + \dots + \alpha_n$.
\item
Let $\preceq$ be a total well-ordering on $\mathbb{N}^n$. Then, given any nonzero element $r = \sum_{\alpha\in\mathbb{N}^n} r_\alpha \bm{x}^\alpha \in R$, where all $r_\alpha\in C$, we make the following definitions.
\begin{itemize}
\item The \emph{support} of $r$ is $\supp(r) = \{ \alpha : r_\alpha \neq 0\} \subseteq \mathbb{N}^n$.
\item The \emph{least monomial} of $r$ is $\LM_\preceq (r) = \min_\preceq (\supp(r))$. This will be denoted $\LM(r)$ when the ordering $\preceq$ is understood. (This is a minor abuse of terminology, as $\LM(r)$ is in fact an \emph{exponent} $\alpha$ rather than a \emph{monomial} $\bm{x}^\alpha$, but we blur the distinction between the two here for convenience.)
\item If $\LM(r) = \alpha$, the \emph{least coefficient} of $r$ is $\LC(r) = r_\alpha\in \iota(k)^\times$, and its \emph{least term} is $\LT(r) = r_\alpha \bm{x}^\alpha$.
\end{itemize}
When it is useful, we will set $\LM(0) = \infty$, a symbol with the property that $\alpha\prec\infty$ for all $\alpha\in\mathbb{N}^n$.
\end{enumerate}
\end{defn}

\textbf{Warnings.}

\begin{enumerate}
\item In many of our rings, there will not always be a unique choice of $C$, and so the \emph{support} might be better named the \emph{$C$-support}. Depending on the choice of $C$, the support may or may not behave well with respect to the arithmetic in $R$: see \S \ref{subsec: support calculations} below. After some initial setup, this extra flexibility will be useful.
\item As we will be dealing with power series rings, the definitions of $\LM$, $\LC$ and $\LT$ above are those appropriate to power series rings \cite[p. 525]{BecWei93}, the \emph{reverse} of the standard definitions found in the theory of commutative polynomial rings \cite[Chapter 2, \S 2, Definition 7]{CoxLitOSh07}.
\end{enumerate}

The following definition is standard in the commutative theory, though is usually called a \emph{monomial ordering} rather than an \emph{additive total ordering}. We require extra hypotheses on an additive total ordering to ensure that it interacts nicely with the structure of $R$, and so we reserve the former phrase for a stronger notion, to be discussed in \S \ref{subsec: monomial orders and compatibility with multiplication} below.

\begin{defn}
Let $\prec$ be a total ordering on $\mathbb{N}^n$. We will say that $\prec$ is \emph{additive} if, for all $\alpha,\beta,\gamma\in\mathbb{N}^n$,

\begin{enumerate}[label=(\roman*)]
\item if $\alpha \neq \bm{0}$ then $\bm{0} \prec \alpha$,
\item  if $\alpha \prec \beta$ then ${\alpha+\gamma} \prec {\beta+\gamma}$.
\end{enumerate}
\end{defn}

\begin{rks}\label{rks: well-ordering and refines div}
Let $\preceq$ be an additive total ordering. Then:

\begin{enumerate}[label=(\roman*)]
\item $\preceq$ is a well-ordering \cite[Lemma 1.3.2]{pangalos}.
\item $\preceq$ refines the partial order $\leq_\div$: if $\alpha \leq_\div \beta$, then $\beta-\alpha \in\mathbb{N}^n$, so $\bm{0} \preceq \beta-\alpha$ and hence $\alpha \preceq \beta$.
\end{enumerate}
\end{rks}

\begin{exs}\label{exs: orderings}
The following are well-known additive total orderings \cite[Chapter 2, \S 2, Definitions 3, 5 and 6]{CoxLitOSh07}, and we repeat their definitions to fix notation.

\begin{enumerate}
\item The \emph{lexicographic order} $\prec \; = \; \prec_{\mathrm{lex}}$ is defined by the property
$$\alpha \preceq \beta \Leftrightarrow \alpha = \beta \text{ or the leftmost nonzero entry of } \beta-\alpha \in \mathbb{Z}^n \text{ is positive}.$$
\item The \emph{degree} (or \emph{graded}) \emph{lexicographic order} $\prec \; = \; \prec_{\mathrm{deglex}}$ is defined by the property
$$\alpha \preceq \beta \Leftrightarrow \begin{cases}
|\alpha| < |\beta| &\text{ or}\\
\alpha = \beta &\text{ or} \\
|\alpha| = |\beta| &\text{ and the leftmost nonzero entry of } \beta-\alpha \in \mathbb{Z}^n \text{ is positive}.
\end{cases}$$
\item The \emph{degree} (or \emph{graded}) \emph{reverse lexicographic order} $\prec \; = \; \prec_{\mathrm{degrevlex}}$ is defined by the property
$$\alpha \preceq \beta \Leftrightarrow \begin{cases}
|\alpha| < |\beta| &\text{ or}\\
\alpha = \beta &\text{ or} \\
|\alpha| = |\beta| &\text{ and the rightmost nonzero entry of } \beta-\alpha \in \mathbb{Z}^n \text{ is negative}.
\end{cases}$$
\end{enumerate}

All three of these induce the same ordering $x_n \prec x_{n-1} \prec \dots \prec x_1$ on the variables.
\end{exs}

\subsection{Ordinals}

 For the convenience of the reader, we recall some set-theoretic notions. An overview can be found in \cite[\S\S 4.2--4.3]{Cie97} or \cite[Chapter 7]{End77}.

\begin{defn}
Let $X$ be a set and $\preceq$ be a fixed well-ordering on $X$. 

\begin{enumerate}[label=(\roman*)]
\item Take $x,y\in X$. If $x\prec y$, but there exists no $z\in X$ such that $x\prec z\prec y$, then we will say that $y$ is the \emph{successor} of $x$, or equivalently $x$ is the \emph{predecessor} of $y$.
\item If $x\in X$ has a predecessor, then $x$ is called a \emph{successor ordinal}. Otherwise, $x$ is called a \emph{limit ordinal}.
\item Take a subset $Y\subseteq X$. Then $x\in X$ is an \emph{upper bound for $Y$} if $y\preceq x$ for all $y\in Y$. If $Y$ has an upper bound, we will say it is \emph{bounded above}.
\item Suppose $Y\subseteq X$ is bounded above. Then its \emph{supremum} $\sup(Y)$ is its least upper bound, $\sup(Y) := \min_{\preceq} \{x\in X: y\preceq x \text{ for all } y\in Y\}$.
\end{enumerate}
\end{defn}

\begin{rk*}
Suppose $Y\subseteq X$ is bounded above (and nonempty, to avoid trivialities). Note that, if $Y$ is finite, then we always have $\sup(Y) \in Y$.

Even if $Y$ is infinite: if $\sup(Y)$ is a successor ordinal, then $\sup(Y)\in Y$. Indeed, by definition, $\sup(Y)$ is the minimal element $s\in X$ such that $y\preceq s$ for all $y\in Y$: in particular, if $x$ denotes the predecessor in $X$ of $s$, then $x$ is \emph{not} the supremum of $Y$, so by definition there must exist an element $y\in Y$ such that $y\preceq s$ and $y\succ x$. This element must be $s$.

On the other hand, if $Y$ is infinite and $\sup(Y)$ is a limit ordinal, we need not have $\sup(Y)\in Y$.
\end{rk*}

\begin{lem}\label{lem: replace indexing set by N}
Suppose that $Y\subseteq X$ is bounded above and $\sup(Y) = \alpha$ is a limit ordinal which is not an element of $Y$. Suppose also that $Y$ is countably infinite. Then there is an $\mathbb{N}$-indexed sequence $y_0 \prec y_1 \prec \dots \in Y$ such that $\sup\{y_0, y_1, \dots\} = \alpha$.
\end{lem}

\begin{proof}
As $Y$ is countably infinite, we may choose a bijection $\varphi:\mathbb{N}\to Y$. Let $i_0 = 0$, and for all $m\geq 1$, let $i_m$ be the minimal natural number satisfying $i_m > i_{m-1}$ and $\varphi(i_m) \succ \varphi(i_{m-1})$ (which exists as $\varphi(i_{m-1})$ is not the supremum of $Y$). Take $y_\ell = \varphi(i_\ell)$ for all $\ell \geq 0$.

Now $\alpha$ is an upper bound for $\{y_\ell\}$, so it remains only to show that $\sup_\ell\{y_\ell\} \succeq \alpha$. But for all $\beta \prec \alpha$, there exists $y\in Y$ such that $\beta \prec y$, say $y = \varphi(j)$; and there exists $m$ such that $i_m \geq j$, which by construction implies that $\varphi(i_m) \succeq \varphi(j)$ and hence $\beta \prec y_m$.
\end{proof}

\subsection{The topology on $R$}\label{subsec: topology}

Let $R$ be a ring with maximal ideal $\mathfrak{m}$ satisfying $\bigcap_{n\in\mathbb{N}} \mathfrak{m}^n = 0$. (This is satisfied by our rings of interest by Property \ref{props: iterated SPSRs}(iii).) Give $R$ the (separated) $\mathfrak{m}$-adic filtration in the sense of \cite[Chapter I, Example 2.3B]{LVO} (see also \cite[\S 2]{venjakob} or \cite[3.1]{letzter-noeth-skew}), or equivalently the $\mathfrak{m}$-adic norm induced by this filtration as in \cite[Lemma 6.5]{DDMS}. We will further assume that $R$ is complete with respect to the topology given by this filtration (or norm).

Explicitly, we can understand convergence in $R$ as follows:

\begin{defn}\cite[Definition 6.8(i)]{DDMS}
Let $T$ be a countably infinite indexing set, and let $(r_\alpha)_{\alpha\in T}$ be a family of elements of $R$. Then $(r_\alpha)$ \emph{converges to} $r\in R$ if, for all $N > 0$, there exists a finite subset $T'\subseteq T$ such that $r - r_\alpha\in \mathfrak{m}^N$ for all $\alpha\in T\setminus T'$.
\end{defn}

As $R$ is complete in the sense of \cite[Definition 6.2]{DDMS}, we can rephrase this in terms of Cauchy sequences.

\begin{lem}\label{lem: cauchy criterion} 
Let $T$ be a countably infinite indexing set, and let $(r_\alpha)_{\alpha\in T}$ be a family of elements of $R$. Then $(r_\alpha)$ converges in $R$ if and only if, for all $N > 0$, there exists a finite subset $T'\subseteq T$ such that $r_\alpha - r_\beta\in \mathfrak{m}^N$ for all $\alpha, \beta\in T\setminus T'$. \qed
\end{lem}

We will use this to prove the below proposition, which crucially uses the fact that $k$ is finite.

\begin{propn}\label{propn: convergent subsequences in lex ordering}
Let $T\subseteq \mathbb{N}^n$ be an (infinite) indexing set whose supremum is a limit ordinal $\alpha\not\in T$, and let $(r_\beta)_{\beta\in T}$ be a family of elements of $R$. Then there exists $T'\subseteq T$ satisfying $\sup(T') = \alpha$ such that the subfamily $(r_\beta)_{\beta\in T'}$ converges.
\end{propn}

\begin{proof}
By Lemma \ref{lem: replace indexing set by N}, we can assume without loss of generality that $T = \{\beta_0, \beta_1, \dots\}$, where $\beta_0 \prec \beta_1 \prec \dots$, i.e. $T$ is $\mathbb{N}$-indexed and increasing. Consider the family of maps $\varphi_i: T\to R/\mathfrak{m}^{i+1}$ defined by $\varphi_i(\beta) = r_\beta + \mathfrak{m}^{i+1}$.

As $R/\mathfrak{m} = k$ is finite, there must be some $a_0\in k$ such that $\varphi_0^{-1}(a_0)$ is infinite and has supremum $\alpha$. Fix such an $a_0$, and let $T_0 = \varphi^{-1}(a_0)$ and $\gamma_0 = \min(T_0)$.

Now, for all $i\geq 1$, choose inductively $a_i\in R/\mathfrak{m}^{i+1}$ such that
\begin{itemize}[noitemsep]
\item $a_i + \mathfrak{m}^i = a_{i-1}$,
\item $T_i := (\varphi_i^{-1}(a_i)\cap T_{i-1}) \setminus \{\gamma_0, \dots, \gamma_{i-1}\}$ is infinite and has supremum $\alpha$,
\end{itemize}

and set $\gamma_i := \min(T_i)$. Finally, set $T' = \{\gamma_0, \gamma_1, \dots\}$, and note that $\gamma_0 \prec \gamma_1 \prec \dots$. Then, by construction, the family $(r_\beta)_{\beta\in T'}$ is Cauchy, and so converges by Lemma \ref{lem: cauchy criterion}; and as $T'$ is an infinite subsequence of $T$, it also has supremum $\alpha$ as required.
\end{proof}

\section{First calculations}

\subsection{Setup, notation and conventions}

\begin{setup}\label{setup}
We will construct iterated local skew power series rings $R$ over a base ring $A$ in one of two ways as follows. It will be convenient to adopt similar notation for each, so we spell out both cases below.
\end{setup}

\begin{enumerate}
\item[\textbf{(A)}] Let $A = k$ be a finite field of characteristic $p$, and set $R_0 = A$ with maximal ideal $\mathfrak{m}_0 = 0$. Then, for each $1\leq i\leq n$, proceed inductively as follows:
\begin{itemize}
\item Assume that the local rings $(R_0, \mathfrak{m}_0) \subseteq \dots \subseteq (R_{i-1}, \mathfrak{m}_{i-1})$ have all been defined. Take a local, $A$-linear skew derivation $(\sigma_i, \delta_i)$ on $R_{i-1}$. Set $R_i = R_{i-1}[[x_i; \sigma_i, \delta_i]]$, and call its maximal ideal $\mathfrak{m}_i$.
\end{itemize}
Finally, set $R = R_n$, so that overall we have $R = A[[x_1]][[x_2;\sigma_2,\delta_2]]\dots [[x_n;\sigma_n,\delta_n]].$
\item[\textbf{(A${}_\triangle$)}] As in \textbf{(A)}, but with the additional assumption that each skew derivation $(\sigma_i, \delta_i)$ is \emph{triangular} with respect to the sequence $(R_0, \mathfrak{m}_0) \subseteq \dots \subseteq (R_{i-1}, \mathfrak{m}_{i-1})$.
\item[\textbf{(B)}] Let $A$ be a complete discrete valuation ring of characteristic $0$, with residue field $k$ a finite field of characteristic $p$. Set $R_1 = A$ with maximal ideal $\mathfrak{m}_1$ generated by a fixed uniformiser $x_1$. Then, for each $2\leq i\leq n$, proceed inductively as follows:
\begin{itemize}
\item Assume that the local rings $(R_1, \mathfrak{m}_1) \subseteq \dots \subseteq (R_{i-1}, \mathfrak{m}_{i-1})$ have all been defined. Take a local, $A$-linear skew derivation $(\sigma_i, \delta_i)$ on $R_{i-1}$ which is triangular with respect to this sequence. Set $R_i = R_{i-1}[[x_i; \sigma_i, \delta_i]]$, and call its maximal ideal $\mathfrak{m}_i$.
\end{itemize}
Finally, set $R = R_n$, so that overall we have $R = A[[x_2]][[x_3;\sigma_3,\delta_3]]\dots [[x_n;\sigma_n,\delta_n]].$
\item[\textbf{(B${}_\triangle$)}] As in \textbf{(B)}, but with the additional assumption that each skew derivation $(\sigma_i, \delta_i)$ is \emph{triangular} with respect to the sequence $(R_1, \mathfrak{m}_1) \subseteq \dots \subseteq (R_{i-1}, \mathfrak{m}_{i-1})$.
\end{enumerate}

Cases \textbf{(A${}_\triangle$)} and \textbf{(B${}_\triangle$)} will be the primary cases of interest to us, but we will sometimes work in the more general cases \textbf{(A)} and \textbf{(B)}.

\begin{propn}\label{propn: uniqueness of representation}
In either case, let $C = \iota(k)$ be an arbitrary section of the natural quotient map $A\to k$. Then every element $r\in R$ can be written uniquely as
$$r = \sum_{\alpha\in\mathbb{N}^n} r_\alpha \bm{x}^\alpha$$
for some choice of coefficients $r_\alpha\in C$. Conversely, given any choice of coefficients $r_\alpha\in C$, the sum $\sum_{\alpha\in\mathbb{N}^n} r_\alpha \bm{x}^\alpha$ converges in $R$.
\end{propn}

\begin{proof}
Combine Property \ref{props: iterated SPSRs}(i) and (in case \textbf{(B)}) Lemma \ref{lem: standard form in A}.
\end{proof}

\begin{defn}\label{defn: standard form in R}
As in Definition \ref{defn: standard form in A}, we will say that the expression $r = \sum_{\alpha\in\mathbb{N}^n} r_\alpha \bm{x}^\alpha\in R$ is in \emph{standard form} (with respect to our choices of $\iota$ and, in case \textbf{(B)}, of uniformiser $x_1$) if the coefficients $r_\alpha$ lie in $C$.
\end{defn}

Suppose that we are in the situation of Proposition \ref{propn: uniqueness of representation}: let $C$ be an arbitrary section of the map $A\to k$, and suppose $r\in R$ has been written uniquely as $r = \sum_\alpha r_\alpha \bm{x}^\alpha$ for some $r_\alpha\in C$. We can immediately interpret Property \ref{props: iterated SPSRs}(ii) as follows:

\begin{lem}\label{lem: standard form of powers of m}
$r\in \mathfrak{m}^\ell$ if and only if $r_\alpha = 0$ for all $|\alpha| < \ell$.\qed
\end{lem}

\textbf{Notation and conventions.}

\begin{enumerate}
\item If the pair $A\subseteq R$ satisfies \textbf{(A)} or \textbf{(B)} above, we will denote this by writing $R = A[[\bm{x}; \bm{\sigma}, \bm{\delta}]]^n$. If the pair $A\subseteq R$ satisfies the more stringent conditions \textbf{(A${}_\triangle$)} or \textbf{(B${}_\triangle$)}, we will write $R = A[[\bm{x}; \bm{\sigma}, \bm{\delta}]]^n_\triangle$ instead.
\item If $R = A[[\bm{x}; \bm{\sigma}, \bm{\delta}]]^n$, we will continue to use the notation $x_1, \dots, x_n$ for the ordered list of formal variables and to write $(R_i, \mathfrak{m}_i)$ for the local subrings and $k$ for the residue field.
\item From now on, we will write $\iota: k \to A$ for the unique \emph{multiplicative} section of the natural quotient map $A\to k$. (This exists by \cite[Chapter II, \S 4, Proposition 8]{Ser79}, as we are assuming throughout that $k$ is a finite field, and hence a perfect field of positive characteristic.) In cases \textbf{(A)} and \textbf{(A${}_\triangle$)}, where $A = k$ is a field, then $\iota$ is just the identity map; in cases \textbf{(B)} and \textbf{(B${}_\triangle$)}, it is not necessarily an additive map.
\end{enumerate}

\begin{rk}\label{rk: choice of section}
We will always take $C = \iota(k)$, for this multiplicative $\iota$, throughout the remainder of this paper. This choice of $C = \iota(k)$ may not be necessary, but is a significant simplification to our calculations from \S \ref{subsec: support calculations} onwards. Note, in particular, that $-1_A = \iota(-1_k) \in\iota(k)$, and so $a\in \iota(k)$ if and only if $-a\in\iota(k)$.
\end{rk}

We now consider sequences in $R$, and rephrase a result of \S \ref{subsec: topology} in the language of monomials. In the following proposition, we will take $\prec \; = \; \prec_{\mathrm{deglex}}$.

\subsection{Support under addition and subtraction}\label{subsec: support calculations}

Let $R = A[[\bm{x}; \bm{\sigma}, \bm{\delta}]]^n$. In this section, we present technical calculations to unify cases \textbf{(A)} and \textbf{(B)} so that we can handle both cases together.

In case \textbf{(A)}, addition and subtraction is termwise: that is, if $r = \sum_\alpha r_\alpha \bm{x}^\alpha$ and $s = \sum_\alpha s_\alpha \bm{x}^\alpha$ for $r_\alpha, s_\alpha\in k$, then $r \pm s = \sum_{\alpha} (r_\alpha \pm s_\alpha) \bm{x}^\alpha$, and $r_\alpha \pm s_\alpha \in k$. However, recall that our standard form of elements in case \textbf{(B)} (Definition \ref{defn: standard form in R}) relies on the choice of coefficient set $\iota(k)$, and this will in general be badly-behaved with respect to addition and subtraction.

\begin{ex}\label{ex: support is badly behaved under addition}
Let $R = \mathbb{Z}_3$, so that $\iota(\mathbb{F}_3) = \{0, -1, 1\}$. Fix the uniformiser $\pi = 3$. Then $1 + 1 = 2$ in this ring. However, $1 = 1\pi^0$ is already in standard form, but $2$ has standard form $\pi - 1 = 1\pi^1 + (-1)\pi^0$, and so $\supp(1) = \{0\}$ but $\supp(2) = \{0, 1\}$.
\end{ex}

For this reason we need several technical results that ensure that the necessary calculations are tractable. Many of these are far easier to prove in case \textbf{(A)}, so we omit their proofs in this case, and focus only on case \textbf{(B)}.

For the remainder of this section, unless stated otherwise, assume we are in case \textbf{(B)}. Recall that $x_1$ is the uniformiser of $A$, and write $v$ for the normalised $x_1$-adic valuation function on $A$ as in \S \ref{subsec: CDVRs}.

\begin{lem}\label{lem: support of sum of monomials}
Let $a, b\in A$ and $\gamma\in \mathbb{N}^n$ be arbitrary. Then $(a \pm b)\bm{x}^\gamma$ has support contained in $\gamma + (\mathbb{N}\times \{0\}^{n-1})$.
\end{lem}

\begin{proof}
Since $a \pm b\in A$, we may write it in standard form as $c_0 + c_1 x_1 + \dots$, where the coefficients $c_i\in \iota(k)$. Hence
$$(a \pm b)\bm{x}^\gamma = \sum_{i\geq 0} c_i x_1^i \bm{x}^{\gamma}.$$
But $x_1^i \bm{x}^\gamma = \bm{x}^{\gamma + i\bm{e}_1}$, where $\bm{e}_1 = (1,0,\dots,0)\in \mathbb{N}^n$. In particular, we can calculate that $\supp((a \pm b)\bm{x}^\gamma) = \{\gamma + i\bm{e}_1 : c_i \neq 0\}$.
\end{proof}

The following are essentially inductive versions of Lemma \ref{lem: support of sum of monomials}, so we omit their proofs. Note that the assumption ``$v(a_i)\to \infty$" is necessary for the sum in part (i) to converge.

\begin{lem}\label{lem: support of sum}
$ $

\begin{enumerate}[label=(\roman*)]
\item Let $a_0, a_1, \dots \in A$ and $\gamma\in\mathbb{N}^n$. Suppose that $v(a_i) \to \infty$ as $i\to \infty$. Then we have $\supp(\sum_{i=0}^\infty a_i \bm{x}^\gamma)\subseteq \gamma + (\mathbb{N}\times \{0\}^{n-1})$.
\item Let $r,s\in R$. Then $\supp(r\pm s) \subseteq (\supp(r) \cup \supp(s)) + (\mathbb{N}\times \{0\}^{n-1})$.
 \qed
\end{enumerate}
\end{lem}

Note that in case \textbf{(A)}, we have $\supp(r+ s) \subseteq \supp(r) \cup \supp(s)$, but due to Example \ref{ex: support is badly behaved under addition} we cannot conclude this in case \textbf{(B)}. However, the above lemma is strong enough to conclude the following:

\begin{propn}\label{propn: LM is filtered}
Fix any additive total ordering $\preceq$, and let $r,s\in R$. Then $\LM = \LM_\preceq$ satisfies the following properties:
\begin{enumerate}[label=(\roman*)]
\item $\LM (r\pm s) \succeq \min\{\LM(r), \LM(s)\}$,
\item if $\LM(r) \neq \LM(s)$, then $\LM (r\pm s) = \min\{\LM(r), \LM(s)\}$.
\end{enumerate}
\end{propn}

\begin{proof}
$ $

\begin{enumerate}[label=(\roman*)]
\item This follows from Remark \ref{rks: well-ordering and refines div}(ii) and Lemma \ref{lem: support of sum}(ii), noting that $\alpha + (\mathbb{N}\times \{0\}^{n-1})$ has unique $\leq_\div$-minimal element $\alpha$.
\item Assume without loss of generality that $\LM(r) \prec \LM(s)$. Then note that $r = (r\pm s)\mp s$, and apply part (i) twice to see that $$\LM(r) = \LM((r\pm s)\mp s) \succeq \min\{\LM(r\pm s), \LM(s)\} \succeq \min\{\LM(r), \LM(s)\}.$$ Putting all this information together yields $\LM(r) \succeq \LM(r\pm s) \succeq \LM(r)$.\qedhere
\end{enumerate}
\end{proof}

To determine whether two elements $r,s\in R$ are equal, we may wish to ask whether $\supp(r-s)$ is empty. The following lemma implies that we do not need to check all of $(\supp(r) \cup \supp(s)) + (\mathbb{N}\times \{0\}^{n-1})$, and need only check $\supp(r) \cup \supp(s)$, even in case \textbf{(B)}:

\begin{lem}\label{lem: checking equality on support}
$r = s$ if and only if $\supp(r-s) \cap (\supp(r) \cup \supp(s)) = \varnothing$.
\end{lem}

\begin{proof}
The forward implication is clear, because if $r = s$ then $\supp(r-s) = \supp(0) = \varnothing$. For the reverse implication, suppose $r\neq s$, and write $r = \sum_\alpha r_\alpha \bm{x}^\alpha$ and $s = \sum_\alpha s_\alpha \bm{x}^\alpha$, where $r_\alpha, s_\alpha\in\iota(k)$. Then $r-s = \sum_\alpha (r_\alpha - s_\alpha) \bm{x}^\alpha \neq 0$.

Fix any additive total ordering $\preceq$ on $\mathbb{N}^n$, and let $\beta$ be the $\preceq$-minimal exponent such that $r_\beta \neq s_\beta$. This implies in particular that $r_\beta \not\equiv s_\beta \bmod J(A)$, as the elements of $\iota(k)$ are distinct mod $J(A)$. It follows that $r_\beta - s_\beta \not\equiv 0 \mod J(A)$. Hence $\LM_\preceq(r-s) = \beta\in \supp(r) \cup \supp(s)$.
\end{proof}

\subsection{Support under the action of the skew derivation}

Following Proposition \ref{propn: uniqueness of representation}, suppose that the element $r\in R$ has unique representation
\begin{align*}
r = \sum_{\alpha\in \mathbb{N}^n} r_\alpha \bm{x}^\alpha
\end{align*}
with coefficients $r_\alpha\in \iota(k)$.

\begin{propn}\label{propn: sigma and delta of x_r}
Suppose that we are in either case \textbf{(A${}_\triangle$)} or case \textbf{(B${}_\triangle$)}. Let $1\leq r < s\leq n$, and write $\sigma = \sigma_s$, $\delta = \delta_s$. Then there exist constants $q\in \iota(k)^\times$ and $b_t, c_\gamma, c'_\gamma \in \iota(k)$ such that
$$\sigma(x_r) = qx_r + \sum_{t=1}^{r-1} b_t x_t + \sum_\gamma c_\gamma \bm{x}^\gamma, \quad \qquad \delta(x_r) = \sum_\gamma c'_\gamma \bm{x}^\gamma,$$
where in both sums $\gamma$ runs over all elements of $\mathbb{N}^r\times \{0\}^{n-r}$ of total degree at least 2.
\end{propn}

\begin{proof}
In both cases, the \emph{triangularity} condition tells us that, for all $1 \leq r < s \leq n$, we have $x_s x_r = \sigma_s(x_r)x_s + \delta_s(x_r)$ where $\sigma_s(x_r) \in \mathfrak{m}_r$ and $\delta_s(x_r) \in \mathfrak{m}_r^2$. Almost everything claimed now follows immediately from the fact that 
$$\displaystyle \mathfrak{m}_r = \left\{ \sum_{\gamma\in\mathbb{N}^n} c_\gamma \bm{x}^\gamma : c_\gamma \neq 0 \implies \gamma\in \mathbb{N}^r\times \{0\}^{n-r} \text{ and } |\gamma| \geq 1 \right\},$$ 
$$\displaystyle \mathfrak{m}_r^2 = \left\{ \sum_{\gamma\in\mathbb{N}^n} c_\gamma \bm{x}^\gamma : c_\gamma \neq 0 \implies \gamma\in \mathbb{N}^r\times \{0\}^{n-r} \text{ and } |\gamma| \geq 2 \right\},$$
both of which are Lemma \ref{lem: standard form of powers of m} applied to $\mathfrak{m}_r \subseteq R_r$: it only remains to prove that $q$ is invertible.

But the triangularity assumption implies that $\sigma_s$ restricts to a ring automorphism of $R_j$ for each $1\leq j\leq r$, and so $\sigma_s(\mathfrak{m}_j) = \mathfrak{m}_j$ and $\sigma_s(\mathfrak{m}_j^2) = \mathfrak{m}_j^2$ for all $j$. In particular, writing $V_j$ for the $k$-vector space $\mathfrak{m}_j/\mathfrak{m}_j^2$, we see that $\sigma_s$ is a $k$-linear automorphism of $V_j$. Note that we have a flag of subspaces $V_1 \subset \dots \subset V_r$ on which $\sigma$ acts, that $\dim_k V_j = \dim_k V_{j-1} + 1$ for each $j$, and that $\sigma$ acts via an invertible upper-triangular matrix on $V_r$ with respect to this flag. In particular, the action of $\sigma$ on $V_r/V_{r-1}$ is invertible.
\end{proof}

\subsection{Direct calculation: a three-dimensional worked example}

In this subsection, we will analyse the prime ideal structure of $\mathbb{Z}_p[[x]][[y; \delta]]$, where $\delta(x) = p^2$. We begin with some calculations. Write $v_p$ for the $p$-adic valuation on $\mathbb{Z}_p$.

\begin{lem}\label{lem: p-adic valuation of permutations}
Take some $a, r\in\mathbb{N}$ with $a \geq p^r$. Then
$$v_p\left( \frac{a!}{(a-p^r)!} \right) \geq p^{r-1} + p^{r-2} + \dots + p + 1,$$
with equality if and only if
$$a\not\in \bigcup_{s\geq r+1} \{p^s, p^s+1, \dots, p^s+p^r-1\}.$$
\end{lem}

\begin{proof}
We wish to find the $p$-adic valuation of the product of the numbers $a-p^r+1$, $a-p^r+2$, $\dots, a-1$, $a$. These form a sequence of $p^r$ consecutive integers, and so exactly $p^{r-s}$ of them must be divisible by $p^s$ for each $s = 1, 2, \dots, r-1, r$. Rephrasing this,
$$\begin{cases}
\text{exactly } p^{r-1} - p^{r-2} \text{ of them are divisible by } p \text{ but not } p^2,\\
\text{exactly } p^{r-2} - p^{r-3} \text{ of them are divisible by } p^2 \text{ but not } p^3,\\
\qquad \qquad \qquad \qquad \qquad \quad \qquad \vdots\\
\text{exactly } p - 1 \text{ of them are divisible by } p^{r-1} \text{ but not } p^r,\\
\text{exactly } 1 \text{ of them is divisible by } p^s \text{ for some } s \geq r.
\end{cases}$$
That is, taking $s$ to be the maximal value for which the above is true,
\begin{align*}
v_p\left( \frac{a!}{(a-p^r)!} \right) &= 1(p^{r-1}-p^{r-2}) + 2(p^{r-2}-p^{r-3}) + \dots + (r-1)(p-1) + s\\
&= p^{r-1} + p^{r-2} + \dots + p + (s - r + 1),
\end{align*}
and so the desired inequality holds, with equality exactly when $s = r$.
\end{proof}

Now consider $\mathbb{Z}_p[[x,y]]$ simply as a (topological) module, and write $\partial_x := \partial/\partial x$ and $\partial_y := \partial/\partial_y$ for the (formal, termwise) partial derivatives. In the following lemma, we may take $\preceq$ to be any additive total order on $\mathbb{N}^3$.

\begin{lem}\label{lem: LM of derivatives}
Suppose that $\LM(r) = \LM(p^a x^b y^c) = (a, b, c)$.

\begin{enumerate}[label=(\roman*)]
\item If $v_p(b) = t$, then $\LM(\partial_x^{p^t}(r)) = \LM(\partial_x^{p^t}(p^a x^b y^c)) = (a+v, b-p^t, c)$, where $v = p^{t-1} + \dots + p + 1$.
\item If $v_p(c) = t$, then $\LM(\partial_y^{p^t}(r)) = \LM(\partial_y^{p^t}(p^a x^b y^c)) = (a+v, b, c-p^t)$, where $v = p^{t-1} + \dots + p + 1$.
\end{enumerate}
\end{lem}

\begin{proof}
Parts (i) and (ii) are identical after a change of variables, so we prove (i). Note that $v_p(b) = t$ implies that $b \geq p^t$. We can calculate that
$$\displaystyle \partial_x^{p^t}(p^a x^b y^c) = \dfrac{b!}{(b-p^t)!} p^a x^{b-p^t} y^c.$$
To put this in standard form, we use Lemma \ref{lem: p-adic valuation of permutations}: our assumptions on $b$ imply that $$b \not \in \bigcup_{s\geq t+1} \{p^s, \dots, p^s+p^t-1\},$$ and so
$$v = v_p\left( \dfrac{b!}{(b-p^t)!}\right) = p^{t-1} + \dots + p + 1,$$
and $\LM(\partial_x^{p^t}(p^a x^b y^c)) = (a+v, b-p^t, c)$.

Now take any other $(a', b', c')\in \supp(r)$. By a similar argument, we have $$\supp(\partial_x^{p^t}(p^{a'} x^{b'} y^{c'})) \subseteq (a'+v'+\mathbb{N}, b'-p^t, c')$$ for some $v' \geq v$, and it can be checked using additivity that $(a,b,c) \preceq (a',b',c')$ implies $(a+v, b-p^t, c)\preceq (a'+v'+\mathbb{N}, b'-p^t, c')$.
\end{proof}

\textbf{For the remainder of this section,} we must take $\preceq$ to be the \emph{lexicographic} order.

\begin{ex}\label{ex: 3d over Zp, yx = xy + p^2}
Let $R = \mathbb{Z}_p[[x]][[y; \delta]]$, where $\delta(x) = p^2$. That is, $yx - xy = p^2$, which we may think of as
$$[y, -] = p^2 \partial_x, \qquad \qquad [x, -] = -p^2 \partial_y.$$
Let $P$ be a nonzero prime ideal of $R$. If $p\in P$, then we may quotient by $pR$: we get $\overline{P} \lhd \overline{R} = \mathbb{F}_p[[x,y]]$, and we have reduced to a commutative problem. So \textbf{we assume throughout that} $p\not \in P$.

As we have taken $\preceq$ to be the lexicographic order, if there exists $0\neq r\in P$ such that $\LM(r) = (a,b,c)$ for $a\neq 0$, then \emph{every} monomial in the support of $r$ is divisible by $p$, and so $r$ itself is divisible by $p$. As $P$ is a prime ideal and $p$ is central, we must therefore have $p^{-1}r\in P$.

In particular, we may choose an element $0\neq r\in P$ such that $\LM(r) = (0,b,c)$. We may also assume that $r$ has been chosen so that, for all $0\neq r'\in P$,
\begin{itemize}
\item if $\LM(r') = (0,b',c)$, then $b'\geq b$, and
\item if $\LM(r') = (0,b,c')$, then $c'\geq c$.
\end{itemize}
Since $P$ is assumed two-sided, it is preserved by $[y, -]$ and $[x, -]$, and hence (again using the assumption that $p\not\in P$) it is preserved by $\partial_x$ and $\partial_y$.

So suppose for contradiction that $b \neq 0$: we will apply Lemma \ref{lem: LM of derivatives}(i). Take $t = v_p(b)$, so that $\LM(\partial_x^{p^t}(r)) = (v, b-p^t, c)$, where $v = p^{t-1} + \dots + p + 1$. Hence $p^{-v} \partial_x^{p^t}(r) \in P$ and $\LM(p^{-v} \partial_x^{p^t}(r)) = (0, b - p^t, c)$, a contradiction to our choice of $r$. So we must have $b = 0$, and the same procedure using Lemma \ref{lem: LM of derivatives}(ii) will show that $c = 0$.

But if $\LM(r) = (0,0,0)$, then $r\in R\setminus \mathfrak{m}$, and hence is a unit since $R$ is local. This implies that $P = R$, contradicting our assumption that $p\not\in P$.

In conclusion, we have shown that every nonzero prime ideal of $R$ contains $p$, and hence that $pR$ is the unique prime ideal of $R$ of height 1.
\end{ex}


\section{Monomial orderings}

Continue to fix a representation $R = A[[\bm{x}; \bm{\sigma}, \bm{\delta}]]^n$ throughout, and let $\preceq$ be an additive total ordering on $\mathbb{N}^n$. We will describe several desirable properties that $\preceq$ might satisfy.

\subsection{Compatibility with multiplication}\label{subsec: monomial orders and compatibility with multiplication}

When $R$ is commutative, multiplication of two monomials $\bm{x}^\alpha, \bm{x}^\beta\in R$ and addition of their exponents $\alpha,\beta\in\mathbb{N}^n$ are essentially the same. In general, this is not the case. We will make the following definition as a partial remedy.

\begin{defn}\label{defn: monomial order}
We will say that $\preceq$ is \emph{multiplicative} if $\LM_\preceq (\bm{x}^\alpha \bm{x}^\beta) = \alpha+\beta$ for all $\alpha,\beta\in\mathbb{N}^n$. A multiplicative, additive total order is called a \emph{monomial order}.
\end{defn}

Note that other papers dealing with different noncommutative contexts have different definitions of monomial orders: compare e.g. \cite[Definition 11]{GalLez10}.

When we have a monomial order, the following lemma is the key to calculating with least terms. Note that we are again using the fact that we have chosen $\iota$ to be multiplicative (Remark \ref{rk: choice of section}).

\begin{lem}\label{lem: LTs and LCs}
Suppose that $\preceq$ is a monomial order on $R$.
\begin{enumerate}[label=(\roman*)]
\item If $f,g\in R$ are nonzero elements satisfying $\LM(f) = \varphi$ and $\LM(g) = \gamma$, then we get $\LT(fg) = \LC(f) \LC(g) \LC(\bm{x}^\varphi \bm{x}^\gamma) \bm{x}^{\varphi + \gamma}$.
\item If $f,g,h\in R$ are nonzero, then $\LT(f\LT(gh)) = \LT(fgh) = \LT(\LT(fg)h)$.
\item Let $\alpha,\beta,\gamma\in \mathbb{N}^n$. Then
$$\LC(\bm{x}^\alpha \bm{x}^\beta \bm{x}^\gamma) = \LC(\bm{x}^{\alpha+\beta} \bm{x}^\gamma) \LC(\bm{x}^\alpha \bm{x}^\beta) = \LC(\bm{x}^\alpha \bm{x}^{\beta+\gamma}) \LC(\bm{x}^\beta \bm{x}^\gamma).$$
\end{enumerate}
\end{lem}

\begin{proof}
To prove part (i), simply write $f$ and $g$ explicitly as
$$f = \sum_{\varphi \preceq \alpha} f_\alpha \bm{x}^\alpha, \qquad g = \sum_{\gamma\preceq \beta} g_\beta \bm{x}^\beta,$$
where $f_\alpha, g_\beta \in \iota(k)$ for all $\alpha$ and $\beta$, so that $f_\varphi = \LC(f)$ and $g_\gamma = \LC(g)$ are invertible elements of $\iota(k)$. To multiply $f$ and $g$, we must multiply each term of $f$ with each term of $g$, so take $\alpha\in\supp(f)$ and $\beta\in\supp(g)$: then
$$(f_\alpha \bm{x}^\alpha)(g_\beta \bm{x}^\beta) = f_\alpha g_\beta \bm{x}^\alpha \bm{x}^\beta = f_\alpha g_\beta (\LC(\bm{x}^\alpha \bm{x}^\beta)\bm{x}^{\alpha + \beta} + \dots),$$
where the dots denote $\preceq$-greater terms.

By the multiplicativity assumption on $\preceq$, the coefficient $\LC(\bm{x}^\alpha \bm{x}^\beta)$ is always nonzero. So if $f_\alpha$ and $g_\beta$ are also nonzero, we can read off from the calculation above that the least term of $(f_\alpha \bm{x}^\alpha)(g_\beta \bm{x}^\beta)$ is $f_\alpha g_\beta \LC(\bm{x}^\alpha \bm{x}^\beta)\bm{x}^{\alpha + \beta}$. Finally, as $\preceq$ is additive, $\varphi + \gamma \preceq \alpha + \gamma \preceq \alpha + \beta$, so the least term of $fg$ must be $f_\varphi g_\gamma \LC(\bm{x}^\varphi \bm{x}^\gamma) \bm{x}^{\varphi + \gamma}$ as required.

Part (ii) follows by a similar argument. To prove part (iii), set $f = \bm{x}^\alpha$, $g = \bm{x}^\beta$ and $h = \bm{x}^\gamma$ in part (ii), and use part (i) to simplify the resulting expressions.
\end{proof}

Recall from the Introduction that $\preceq_{\mathrm{deglex}}$ and $\preceq_{\mathrm{degrevlex}}$ are often monomial orders even under relatively weak hypotheses, but that the monomial order of primary interest to us is $\preceq_{\mathrm{lex}}$. The remainder of this section is dedicated to showing that $\preceq_{\mathrm{lex}}$ is a monomial order in the triangular case $R = A[[\bm{x}; \bm{\sigma}, \bm{\delta}]]^n_\triangle$.

\subsection{An intermediate partial order}

In this subsection, we define a partial order $\leq_F$ which refines $\leq_\div$ and can be refined to any of the additive total orders $\preceq_{\mathrm{lex}}$, $\preceq_{\mathrm{deglex}}$ and $\preceq_{\mathrm{degrevlex}}$.

\begin{notn}
$ $

\begin{itemize}
\item Write $\bm{e}_i$ for the $i$th standard basis element $(0, \dots, 0, 1, 0, \dots, 0)\in\mathbb{N}^n$.
\item An \emph{elementary shuffle} is an element $\sigma_{ij} := \bm{e}_i - \bm{e}_j \in \mathbb{Z}^n$ for some $j \geq i$. A \emph{shuffle} is a finite sum of elementary shuffles. The set of shuffles in $\mathbb{Z}^n$ is denoted $\mathcal{S}_n$.
\item Given $\alpha\in\mathbb{N}^n$, we will temporarily write $D(\alpha) = \alpha + \mathbb{N}^n$ and $E(\alpha) = (\alpha + \mathcal{S}_n)\cap \mathbb{N}^n$. If $X\subseteq \mathbb{N}^n$, write $D(X) = \bigcup_{\alpha\in X} D(\alpha)$ and $E(X) = \bigcup_{\alpha\in X} E(\alpha)$.
\end{itemize}
\end{notn}

\begin{lem}\label{lem: ED = DE}
$E(D(\alpha)) \subseteq D(E(\alpha))$.
\end{lem}

\begin{proof}
Write $D^1(\alpha) = \{\alpha + \bm{e}_i \;:\; 1\leq i\leq n\}$ and $E^1(\alpha) = \{\alpha + \sigma_{ij} \;:\; 1\leq i\leq j\leq n\}\cap \mathbb{N}^n$. If $\beta\in D(\alpha)$, then $\beta \in D^1(\dots(D^1(\alpha))\dots)$ (where $D^1$ has been repeatedly applied some finite number of times), and likewise if $\gamma\in E(\beta)$ then $\gamma\in E^1(\dots(E^1(\beta))\dots)$. It will clearly suffice to show that $E^1(D^1(\alpha)) \subseteq D^1(E^1(\alpha))$.

So take $\beta = \alpha + \bm{e}_l\in D^1(\alpha)$ and $\gamma = \beta + \sigma_{ij}\in E^1(\beta)$. We will construct $\beta'\in E^1(\alpha)$ such that $\gamma\in D^1(\beta')$.

\textbf{Case 1.} If $\alpha_j \geq 1$, then we take $\beta' = \alpha + \sigma_{ij} \in \mathbb{N}^n$ and $\gamma = \beta' + \bm{e}_l$.

\textbf{Case 2.} If $\alpha_j = 0$, then since $\gamma_j\geq 0$ we must have $l = j$, so we take $\beta' = \alpha$ and $\gamma = \beta' + \bm{e}_i$.
\end{proof}

Now define $F(X) = D(E(X))$. Lemma \ref{lem: ED = DE} implies that $F(F(X)) = F(X)$.

\newcommand{\mycolor}{white}

\centerline{
\begin{tikzpicture}[scale=0.75]
    \draw [->](-0.5,0) -- (7.8,0);
    \draw [->](0,-0.5) -- (0,5.8);
    \node[] at (7.5,-0.3) {$a$};
    \node[] at (-0.3,5.5) {$b$};

  \foreach \x in {0,...,7} {
    \foreach \y in {0,...,5} {
      \path(0,0)\pgfextra{\ifthenelse{\(\x=1 \AND \y > 2\) \OR \(\x=2 \AND \y > 1\) \OR \(\x=3 \AND \y > 0\) \OR \(\x>3\)}{\def\mycolor{gray!70}}{}}
      node [circle,draw,fill=\mycolor,minimum size=10]  (\x\y) at (\x,\y) {};}}
\end{tikzpicture}
}

\centerline{
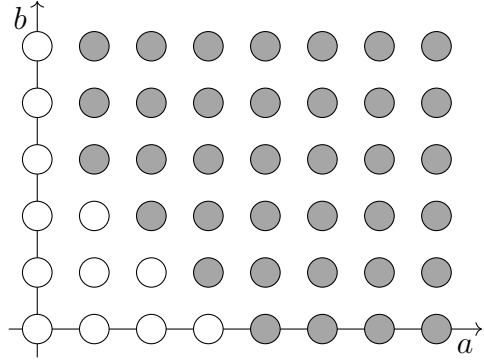
Figure 1: $\mathbb{N}^2 = \{(a,b) : a,b\in\mathbb{N}\}$, with $F(\alpha)$ for $\alpha = (1,3)$ shaded
}

\begin{lem}\label{lem: F is compatible with lex orders}
Let $\preceq$ be any of $\preceq_{\mathrm{lex}}$, $\preceq_{\mathrm{deglex}}$ or $\preceq_{\mathrm{degrevlex}}$. Then $\inf_\preceq (F(\alpha)) = \alpha$.
\end{lem}

\begin{proof}
With notation as in the proof of Lemma \ref{lem: ED = DE}, it suffices to show that, if $\beta\in D^1(\alpha)$ or $\beta\in E^1(\alpha)$, then $\alpha \preceq \beta$, which follows immediately from the definitions in Examples \ref{exs: orderings}.1--3.
\end{proof}

\begin{defn}
Define $\leq_F$ to be the relation on $\mathbb{N}^n$ defined by $\alpha \leq_F \beta \Leftrightarrow \beta\in F(\alpha)$. Note that $\leq_F$ is clearly reflexive, it is transitive by Lemma \ref{lem: ED = DE}, and it is antisymmetric by Lemma \ref{lem: F is compatible with lex orders}, so it is a partial order. 
\end{defn}

A proof similar to that of Lemma \ref{lem: ED = DE} shows that $\leq_F$ also satisfies an additivity property:

\begin{lem}\label{lem: F+ stable under addition}
If $\alpha \leq_F \beta$, then $\alpha+\gamma \leq_F \beta + \gamma$ for all $\gamma\in\mathbb{N}^n$.\qed
\end{lem}

\begin{rk}
$\leq_F$ is \emph{not} a total order on $\mathbb{N}^n$ for any $n > 1$: for instance, on $\mathbb{N}^2$, we have neither $(1,0) \leq_F (0,2)$ nor $(0,2) \leq_F (1,0)$.
\end{rk}

\subsection{Calculating products of monomials}\label{subsec: generalised monomials}

We begin by making the following definition:

\begin{defn}\label{defn: triangular order}
An additive total order $\preceq$ on $\mathbb{N}^n$ is \emph{triangular} if it refines $\leq_F$, i.e. if $\alpha \leq_F \beta$ implies $\alpha \preceq \beta$.
\end{defn}

In this section, we aim to justify this terminology by proving the following theorem, which links the triangularity condition on $R$ (as in cases \textbf{(A${}_\triangle$)} and \textbf{(B${}_\triangle$)} of Setup \ref{setup}) and the triangularity condition on $\preceq$ just defined.

\begin{thm}\label{thm: LM compatibility}
Let $R = A[[\bm{x}; \bm{\sigma}, \bm{\delta}]]^n_\triangle$, and let $\preceq$ be any triangular additive total order. Then $\LM_\prec(\bm{x}^\alpha \bm{x}^\beta) = \alpha+\beta$, and in particular, $\preceq$ is a monomial ordering.
\end{thm}

In particular, once we have proved this, Lemma \ref{lem: F is compatible with lex orders} will show that $\preceq_{\mathrm{lex}}$, $\preceq_{\mathrm{deglex}}$ and $\preceq_{\mathrm{degrevlex}}$ are all monomial orderings.

To prove this, we need some notation for dealing with products $\bm{x}^\alpha \bm{x}^\beta$. This product takes the form $x_{i_1} x_{i_2} \dots x_{i_m}$, and so we will work with the more general context of \emph{generalised monomials}, cf. Remark \ref{rk: standard monomials}. Almost all of the following notation is temporary and will be discarded after this section.

\begin{defn}
$ $

\begin{enumerate}
\item A \emph{generalised monomial} in $R$ is any element $x_{i_1} x_{i_2} \dots x_{i_m}\in R$, where $m$ is any nonnegative integer and each of the indices $i_1, i_2, \dots, i_m$ is chosen from the set $\{1,2,\dots, n\}$. (When $m=0$, we interpret the empty product to be equal to $1_R$.)

The \emph{generalised exponent} of $x_{i_1} x_{i_2} \dots x_{i_m}$ will refer to the $m$-tuple of its indices. For disambiguation, we will write this tuple using square brackets, $[\bm{i}] = [i_1, \dots, i_m]$, and then we will write the element as $\bm{x}_{[\bm{i}]} = x_{i_1} x_{i_2} \dots x_{i_m}$.

\item If $[\bm{i}] = [i_1, \dots, i_m]$ and $[\bm{j}] = [j_1, \dots, j_\ell]$, we will write $[\bm{i}, \bm{j}]$ for their \emph{concatenation} $[\bm{i}, \bm{j}] = [i_1, \dots, i_m, j_1, \dots, j_\ell]$. This allows us to write the product $\bm{x}_{[\bm{i}]} \bm{x}_{[\bm{j}]}$ as $\bm{x}_{[\bm{i},\bm{j}]}$. We will use the same notation for concatenations of several tuples: $[[\bm{i}, \bm{j}], \bm{k}] = [\bm{i}, \bm{j}, \bm{k}]$, etc.

\item The \emph{total degree} of $[\bm{i}] = [i_1, \dots, i_m]$ is $|\bm{i}| = m$.

\item The \emph{reordering} of $[\bm{i}]$ is the element $\rho(\bm{i}) = (\alpha_1, \dots, \alpha_n)\in\mathbb{N}^n$, defined so that $\alpha_r$ is the number of times that $r$ appears among the indices $i_1, \dots, i_m$. The \emph{reordering} of $\bm{x}_{[\bm{i}]}$ is $\bm{x}^{\rho(\bm{i})}$.

Note that, if $R$ is commutative, then each $\bm{x}_{[\bm{i}]}$ is equal to its own reordering. If $R$ is noncommutative, but $[\bm{i}] = [i_1, \dots, i_m]$ is \emph{weakly increasing} ($i_1 \leq \dots \leq i_m$), then $\bm{x}_{[\bm{i}]}$ is simply alternative notation for $\bm{x}^{\rho(i)}$.

\item If $[\bm{i}] = [i_1, \dots, i_m]$ for some $m\geq 2$, then a \emph{transpose} of $[\bm{i}]$ is a tuple $[\bm{i}^\top]$ obtained from $[\bm{i}]$ by swapping the positions of two adjacent entries. That is, for some $1\leq r\leq m-1$, and writing $[\bm{i}_-] = [i_1, \dots, i_{r-1}]$ and $[\bm{i}_+] = [i_{r+2}, \dots, i_m]$ so that $[\bm{i}] = [\bm{i}_-, i_r, i_{r+1}, \bm{i}_+]$, the tuple $[\bm{i}^\top]$ is equal to $[\bm{i}_-, i_{r+1}, i_r, \bm{i}_+]$.
\end{enumerate}
\end{defn}

\begin{notn}
If $\alpha\in\mathbb{N}^n$ is an exponent, write $F^+(\alpha) = F(\alpha) \setminus \{\alpha\}$.

If $[\bm{i}]$ is the tuple of indices of a generalised monomial, write $F(\bm{i})$ for the set of all tuples $[\bm{j}]$ such that $\rho(\bm{j}) \in F(\rho(\bm{i}))$, and similarly $F^+(\bm{i})$.

On generalised monomials, the relation $[\bm{j}] \in F(\bm{i})$ is no longer antisymmetric, so we will avoid writing $\leq_F$.
\end{notn}

\begin{lem}\label{lem: effect of single transpose}
Let $[\bm{i}]$ be arbitrary of total degree at least 2, and $[\bm{i}^\top]$ any transpose of $[\bm{i}]$. Then there exist constants $q\in \iota(k)^\times$ and $a_{[\bm{j}]}\in \iota(k)$ for all $[\bm{j}] \in F^+(\bm{i})$ such that
$$\bm{x}_{[\bm{i}]} - q\bm{x}_{[\bm{i}^\top]} = \sum_{[\bm{j}] \in F^+(\bm{i})} a_{[\bm{j}]} \bm{x}_{[\bm{j}]}.$$
\end{lem}

\begin{rk*}
The sum on the right-hand side makes sense for arbitrary choices of coefficients $a_{[\bm{j}]}\in \iota(k)$. Indeed, for each fixed positive integer $d$, there are only finitely many $[\bm{j}]\in F^+(\bm{i})$ of total degree $d$, and hence only finitely many of these terms are nonzero modulo $\mathfrak{m}^{d+1}$ by Lemma \ref{lem: standard form of powers of m}.
\end{rk*}

In the following proof, write $\bm{e}_i$ for the $i$th standard basis element $(0, \dots, 0, 1, 0, \dots, 0)\in\mathbb{N}^n$. 

\begin{proof}
We begin with the case $[\bm{i}] = [s,r]$, where without loss of generality we may assume that $s > r$, so that $[\bm{i}^\top] = [r,s]$ and $\rho(\bm{i}) = \bm{e}_r + \bm{e}_s$. Write $\sigma = \sigma_s$ and $\delta = \delta_s$: then
$$\bm{x}_{[\bm{i}]} = x_s x_r = \sigma(x_r) x_s + \delta(x_r),$$
and we adopt the notation of Proposition \ref{propn: sigma and delta of x_r}:
\begin{align}
\notag \bm{x}_{[\bm{i}]} &= \left(q x_r +  \sum_{t=1}^{r-1} b_t x_t + \sum_\gamma c_\gamma \bm{x}^\gamma\right) x_s + \left(\sum_\gamma  c'_\gamma \bm{x}^\gamma\right),\text{ i.e.}\\
\label{eqn: s,r expansion} \bm{x}_{[\bm{i}]} - q\bm{x}_{[\bm{i}^\top]}&= \underbrace{\sum_{t=1}^{r-1} b_t \bm{x}^{\bm{e}_t+\bm{e}_s}}_{(1)} + \underbrace{\sum_\gamma c_\gamma \bm{x}^{\gamma + \bm{e}_s}}_{(2)} + \underbrace{\sum_\gamma  c'_\gamma \bm{x}^\gamma}_{(3)},
\end{align}
where in all three sums $\gamma$ ranges over those elements of $\mathbb{N}^r\times \{0\}^{n-r}$ of degree at least 2. It is now easy to check that each monomial on the right-hand side is an element of $F^+(\bm{e}_r + \bm{e}_s) \subseteq F^+(\bm{i}^\top)$:
\begin{enumerate}[label=(\arabic*)]
\item $\bm{e}_t + \bm{e}_s = (\bm{e}_r + \bm{e}_s) + \sigma_{tr} \in E^1(\bm{e}_r + \bm{e}_s)$,
\item choose any $j$ such that $\gamma_j \geq 1$: then $\gamma' = \gamma - \bm{e}_j \in \mathbb{N}^n$, and so $\gamma + \bm{e}_s = (\bm{e}_r + \bm{e}_s) + \sigma_{jr} + \gamma'$ $\in D(E^1(\bm{e}_r + \bm{e}_s))$,
\item choose any $j < j'$ such that $\gamma_j, \gamma_{j'} \geq 1$, or $j = j'$ such that $\gamma_j \geq 2$: then $\gamma' = \gamma - \bm{e}_j - \bm{e}_{j'} \in \mathbb{N}^n$, and so $\gamma = (\bm{e}_r + \bm{e}_s) + \sigma_{jr} + \sigma_{j's} + \gamma' \in D(E^1(E^1(\bm{e}_r + \bm{e}_s)))$.
\end{enumerate}

Now suppose that $[\bm{i}] = [\bm{i}_-, s, r, \bm{i}_+]$. Multiplying equation (\ref{eqn: s,r expansion}) by $\bm{x}_{[\bm{i}_-]}$ on the left and $\bm{x}_{[\bm{i}_+]}$ on the right, we get generalised monomials of the forms
\begin{enumerate}[label=(\arabic*)]
\item $\bm{x}_{[\bm{i}_-]}x_tx_s\bm{x}_{[\bm{i}_+]}$,
\item $\bm{x}_{[\bm{i}_-]} \bm{x}^\gamma x_s \bm{x}_{[\bm{i}_+]}$,
\item $\bm{x}_{[\bm{i}_-]} \bm{x}^\gamma \bm{x}_{[\bm{i}_+]}$,
\end{enumerate}
and so we are left to show that
\begin{enumerate}[label=(\arabic*)] 
\item $\rho(\bm{i}_-) + \bm{e}_t + \bm{e}_s + \rho(\bm{i}_+) \in F^+(\rho(\bm{i}_-) + \bm{e}_r + \bm{e}_s + \rho(\bm{i}_+))$,
\item $\rho(\bm{i}_-) + \gamma + \bm{e}_s + \rho(\bm{i}_+) \in F^+(\rho(\bm{i}_-) + \bm{e}_r + \bm{e}_s + \rho(\bm{i}_+))$,
\item $\rho(\bm{i}_-) + \gamma + \rho(\bm{i}_+) \in F^+(\rho(\bm{i}_-) + \bm{e}_r + \bm{e}_s + \rho(\bm{i}_+))$,
\end{enumerate}
which follow easily from their corresponding numbers above and Lemma \ref{lem: F+ stable under addition}.
\end{proof}

\begin{lem}\label{lem: can reorder a single generalised monomial}
Let $[\bm{i}]$ be arbitrary and $\alpha = \rho(\bm{i})$. Then there exist constants $q\in \iota(k)^\times$ and $a_{[\bm{j}]}\in \iota(k)$ for all $[\bm{j}] \in F^+(\bm{i})$ such that
$$\bm{x}_{[\bm{i}]} - q\bm{x}^\alpha = \sum_{[\bm{j}] \in F^+(\bm{i})} a_{[\bm{j}]} \bm{x}_{[\bm{j}]}.$$
\end{lem}

\begin{proof}
By the same argument that shows that the bubble-sort algorithm terminates, there exists a \emph{finite} sequence of transpositions $[\bm{i}^{\top_1}], [\bm{i}^{\top_1\top_2}], \dots, [\bm{i}^{\top_1\dots \top_e}]$ with $[\bm{i}^{\top_1 \dots \top_e}]$ weakly increasing, so that $\bm{x}_{[\bm{i}^{\top_1 \dots \top_e}]} = \bm{x}^\alpha$. Correspondingly, Lemma \ref{lem: effect of single transpose} gives us a finite list of sums
\refstepcounter{equation}\label{eqn: multiple transposes}
\begin{align}
\tag{\ref{eqn: multiple transposes}.1} \bm{x}_{[\bm{i}]} - q_1\bm{x}_{[\bm{i}^{\top_1}]} &= \sum a_{1,[\bm{j}]} \bm{x}_{[\bm{j}]}\\
\notag \tag{\ref{eqn: multiple transposes}.2} \bm{x}_{[\bm{i}^{\top_1}]} - q_2\bm{x}_{[\bm{i}^{\top_1 \top_2}]} &= \sum a_{2,[\bm{j}]} \bm{x}_{[\bm{j}]}\\
\notag\vdots\qquad&\qquad\vdots\\
\notag \tag{\ref{eqn: multiple transposes}.$e$} \bm{x}_{[\bm{i}^{\top_1 \dots \top_{e-1}}]} - q_e\bm{x}_{[\bm{i}^{\top_1 \dots \top_e}]} &= \sum a_{e,[\bm{j}]} \bm{x}_{[\bm{j}]}
\end{align}
where all sums run over $[\bm{j}]\in F^+(\bm{i})$. (Note that $F^+(\bm{i}) = F^+(\bm{i}^{\top_1}) = \dots$) Now, for each $1\leq N\leq e$, multiply equation (\ref{eqn: multiple transposes}.$N$) by $q_1\dots q_{N-1}$ and add the results to get
$$\bm{x}_{[\bm{i}]} - q'\bm{x}^\alpha = \sum_{[\bm{j}] \in F^+(\bm{i})} a'_{[\bm{j}]} \bm{x}_{[\bm{j}]},$$
where $q' = q_1q_2\dots q_e$.
\end{proof}

\begin{rk*}
In case \textbf{(A${}_\triangle$)}, we can explicitly calculate
$$a'_{[\bm{j}]} = a_{1,[\bm{j}]} + q_1 a_{2,[\bm{j}]} + q_1q_2 a_{3,[\bm{j}]} + \dots + q_1 q_2 \dots q_{e-1} a_{e,[\bm{j}]},$$
though notice that in case \textbf{(B${}_\triangle$)} we cannot: cf. Lemma \ref{lem: support of sum} and Proposition \ref{propn: LM is filtered}.
\end{rk*}

\begin{propn}\label{propn: multiplication is compatible with reordering}
Let $[\bm{i}]$ be arbitrary and $\alpha = \rho(\bm{i})$. Then there exist constants $q\in \iota(k)^\times$ and $a_\beta \in \iota(k)$ for all $\alpha <_F \beta$ such that
$$\bm{x}_{[\bm{i}]} - q\bm{x}^\alpha = \sum_{\alpha <_F \beta} a_{\beta} \bm{x}^\beta.$$
\end{propn}

\begin{proof}
We will proceed by induction on the indexing set $(F(\bm{i}), \preceq)$, where in this proof it will be convenient to take $\preceq \; = \; \preceq_{\mathrm{deglex}}$, as this is a refinement of both $\leq_F$ and the natural $\mathfrak{m}$-adic degree filtration. Then $\preceq$ is order-isomorphic to $(\mathbb{N}, \leq)$, so we can write $F(\bm{i}) = \{\alpha_0, \alpha_1, \alpha_2, \dots\}$ where $\alpha_0 \prec \alpha_1 \prec \dots$ and $\alpha_0 = \alpha$.

Suppose that we have an expression
$$\displaystyle \bm{x}_{[\bm{i}]} = q_0\bm{x}^{\alpha_0} + q_1\bm{x}^{\alpha_1} + \dots + q_{s-1}\bm{x}^{\alpha_{s-1}} + \sum_{\substack{[\bm{j}] \in F(\bm{i}) \\ \alpha_s \preceq \rho(\bm{j})}} a^{(s-1)}_{[\bm{j}]} \bm{x}_{[\bm{j}]},$$
where all $a^{(s-1)}_{[\bm{j}]}\in \iota(k)$. (Lemma \ref{lem: can reorder a single generalised monomial} tells us that such an expression exists for $s = 1$.) For ease of notation, write $a^{(s-1)}_{[\bm{j}]} = a'_{[\bm{j}]}$.

There are finitely many $[\bm{j}]$ satisfying $\rho(\bm{j}) = \alpha_s$, say $[\bm{j}_1], \dots, [\bm{j}_t]$, so we can separate them out:
\begin{align}\label{eqn: separate out alpha_s terms}
\displaystyle \bm{x}_{[\bm{i}]} = q_0\bm{x}^{\alpha_0} + \dots + q_{s-1}\bm{x}^{\alpha_{s-1}} + \left( a'_{[\bm{j}_1]} \bm{x}_{[\bm{j}_1]} + \dots + a'_{[\bm{j}_t]} \bm{x}_{[\bm{j}_t]}\right) + \sum_{\substack{[\bm{j}] \in F(\bm{i}) \\ \alpha_{s+1} \preceq \rho(\bm{j})}} a'_{[\bm{j}]} \bm{x}_{[\bm{j}]}.
\end{align}
Applying Lemma \ref{lem: can reorder a single generalised monomial} to each of the terms $a'_{[\bm{j}_u]} \bm{x}_{[\bm{j}_u]}$ for $1\leq u\leq t$ gives
\refstepcounter{equation}\label{eqn: reorder each generalised term separately}
\begin{align}
\notag \tag{\ref{eqn: reorder each generalised term separately}.$u$}
a'_{[\bm{j}_u]} \bm{x}_{[\bm{j}_u]} = q_{s,u} \bm{x}^{\alpha_s} + \sum_{\substack{[\bm{j}] \in F(\bm{i}) \\ \alpha_{s+1} \preceq \rho(\bm{j})}} b_{[\bm{j}],u} \bm{x}_{[\bm{j}]}.
\end{align}
Now substitute equations (\ref{eqn: reorder each generalised term separately}.1), $\dots$, (\ref{eqn: reorder each generalised term separately}.$t$) into equation (\ref{eqn: separate out alpha_s terms}) to get
\begin{align*}
\displaystyle \bm{x}_{[\bm{i}]} = q_0\bm{x}^{\alpha_0} + \dots + q_{s-1}\bm{x}^{\alpha_{s-1}} + q'_s \bm{x}^{\alpha_s} + \sum_{\substack{[\bm{j}] \in F(\bm{i}) \\ \alpha_{s+1} \preceq \rho(\bm{j})}} c'_{[\bm{j}]} \bm{x}_{[\bm{j}]},
\end{align*}
and (if we are in case \textbf{(B)}) put the constants $q'_s\in A$ and $c'_{[\bm{j}]}\in A$ into standard form, as in Definition \ref{defn: standard form in A}, to get 
\begin{align*}
\displaystyle \bm{x}_{[\bm{i}]} = q_0\bm{x}^{\alpha_0} + \dots + q_{s-1}\bm{x}^{\alpha_{s-1}} + q_s \bm{x}^{\alpha_s} + \sum_{\substack{[\bm{j}] \in F(\bm{i}) \\ \alpha_{s+1} \preceq \rho(\bm{j})}} a^{(s)}_{[\bm{j}]} \bm{x}_{[\bm{j}]}
\end{align*}
where $q_s, a^{(s)}_{[\bm{j}]} \in \iota(k)$.

Hence, by induction, we can reorder the first $s$ monomials for arbitrarily large $s$. But as $s$ increases, the total degree of $[\bm{j}]$ increases, and so the sequence of error terms $\left(\sum a^{(s)}_{[\bm{j}]} \bm{x}_{[\bm{j}]}\right)_{s\in\mathbb{N}}$ converges to zero by Lemma \ref{lem: standard form of powers of m}.
\end{proof}

\textit{Proof of Theorem \ref{thm: LM compatibility}.} Both $\bm{x}^\alpha$ and $\bm{x}^\beta$ are standard monomials, and hence can be written as $\bm{x}_{[\bm{i}]}$ and $\bm{x}_{[\bm{j}]}$ respectively, where $[\bm{i}]$ and $[\bm{j}]$ are weakly increasing and such that $\rho(\bm{i}) = \alpha$ and $\rho(\bm{j}) = \beta$: then the product $\bm{x}^\alpha \bm{x}^\beta$ is equal to $\bm{x}_{[\bm{i},\bm{j}]}$. Proposition \ref{propn: multiplication is compatible with reordering} now implies that there exist $q\in \iota(k)^\times$ and $a_\gamma\in \iota(k)$ such that
$$\bm{x}_{[\bm{i},\bm{j}]} - q\bm{x}^{\alpha+\beta} = \sum_{\alpha+\beta <_F \gamma} a_\gamma \bm{x}^\gamma.$$
In particular, $\alpha+\beta\in \supp(\bm{x}^\alpha \bm{x}^\beta)$, and for every $\gamma \in \supp(\bm{x}^\alpha \bm{x}^\beta) \setminus \{\alpha+\beta\}$, we have $\alpha+\beta <_F \gamma$, hence $\alpha+\beta \prec \gamma$ by Definition \ref{defn: triangular order}. This is what we wanted to show.\qed

We will not use the following corollary in any serious way, but we note for interest:

\begin{cor}\label{cor: associated graded is q-commutative}
Fix a triangular monomial ordering $\preceq$ on $R = A[[\bm{x}; \bm{\sigma}, \bm{\delta}]]^n_\triangle$. Then the function $\LM = \LM_\preceq: R\to \mathbb{N}^n \cup \{\infty\}$ is a rank-$n$ ring filtration, and the associated graded ring is a $q$-commutative polynomial ring (in other words, a multiparameter quantum affine space) over $k$ in $n$ variables.
\end{cor}

\begin{proof}
The fact that $\LM$ is a filtration is just a rephrasing of Theorem \ref{thm: LM compatibility} and Proposition \ref{propn: LM is filtered}. The associated graded ring can be read off from Proposition \ref{propn: multiplication is compatible with reordering}.
\end{proof}

\section{Left ideals and ideal bases}

Let $R = A[[\bm{x}; \bm{\sigma}, \bm{\delta}]]^n$, and fix a monomial order $\preceq$ on $R$ throughout this section.

\begin{notn}
If $S$ is a subset of $R$, we will write $RS$ for the left ideal generated by $S$, and $SR$ for the right ideal. We will also write $\LM(S) = \{\LM(s) : s\in S\setminus \{0\}\} \subseteq \mathbb{N}^n$. Similarly, if $S$ is a tuple of elements of $R$, we will often abuse notation and treat $S$ as a set and write $RS$, $SR$ and $\LM(S)$ for the same objects.

If $T\subseteq \mathbb{N}^n$, we will write $\langle T\rangle = \bigcup_{t\in T} (t + \mathbb{N}^n)$.
\end{notn}

\begin{defn}\label{defn: Groebner basis}
Suppose $I$ is a left ideal of $R$. A \emph{left ideal basis} for $I$ is a finite ordered tuple $G = (g_1, \dots, g_s)$ of nonzero elements $g_i\in I$ such that $I = Rg_1 + \dots + Rg_s$. Note that this implies $\langle \LM(G) \rangle \subseteq \LM(I)$. A left ideal basis $G$ for $I$ is a \emph{Gr\"obner basis} if $\LM(I) = \langle \LM(G)\rangle$.
\end{defn}

\begin{rks*}
$ $

\begin{enumerate}[label=(\roman*)]
\item By Theorem \ref{thm: LM compatibility}, if $I$ is a left ideal of $R$, then $\LM(I) = \langle \LM(I)\rangle$. As in the commutative theory, by Dickson's lemma, there exists a finite subset $T\subseteq \LM(I)$ such that $\LM(I) = \langle T\rangle$. In particular, Gr\"obner bases always exist: if $\LM(I) = \langle T\rangle$, then for any $\alpha\in T$ there exists $r_\alpha \in I$ such that $\LM(r_\alpha) = \alpha$, and the set $G = \{r_\alpha : \alpha\in T\}\subseteq I$ satisfies $\LM(G) = T$ as required.
\item Also as in the commutative theory, it will turn out that remainders on right-division by $G$ are unique regardless of the ordering of the elements of $G$. However, in this paper, it will be convenient to insist throughout that $G$ remain ordered.
\end{enumerate}
\end{rks*}

\subsection{The division algorithm and left ideal membership}

In the following theorem, we will use the following notation. If $F = (f_1, \dots, f_s)$ is an ordered tuple of nonzero elements of $R$, define the following subsets of $\mathbb{N}^n$:
$$\Delta_{F,r} = \left(\LM(f_r) + \mathbb{N}^n\right) \setminus \left( \bigcup_{i=1}^{r-1} \Delta_{F,i}\right) \;\;\; (1\leq r\leq s), \qquad \qquad \overline{\Delta}_F = \mathbb{N}^n \setminus \left(\bigcup_{i=1}^{s} \Delta_{F,i}\right).$$

\begin{thm}[Multivariate right-division]\label{thm: division theorem}
Let $F = (f_1, \dots, f_s)$ be an ordered tuple of nonzero elements of $R$. Then every $f\in R$ can be written uniquely as $f = q_1 f_1 + \dots + q_s f_s + r$, where $q_1, \dots, q_s, r\in R$ are elements satisfying
$$\begin{cases}
\supp(q_i) + \LM(f_i) \subseteq \Delta_{F,i} \text{ for all } 1\leq i\leq s, \text{ and}\\\supp(r) \subseteq \overline{\Delta}_F.\end{cases}$$
Moreover, $\LM(f) \preceq \LM(q_i f_i)$ for all $1\leq i\leq s$.
\end{thm}

\begin{proof}
We perform transfinite induction on $(\mathbb{N}^n, \preceq)$: in particular, for every $\alpha\in\mathbb{N}^n$, we will define elements $g(\alpha), q_1(\alpha), \dots, q_s(\alpha), r(\alpha)\in R$ satisfying
\begin{enumerate}[label=(\arabic*${}_\alpha$)]
\item $f = g(\alpha) + q_1(\alpha)f_1 + \dots + q_s(\alpha)f_s + r(\alpha)$,
\item $\alpha\preceq\LM(g(\alpha))$, and
\item for all $\beta\prec\alpha$, $\begin{cases}
\beta\preceq \LM(q_i(\alpha) - q_i(\beta)) + \LM(f_i) \text{ for all } 1\leq i\leq s,\\
\beta\preceq \LM(r(\alpha) - r(\beta)),\end{cases}$
\item $\begin{cases}
\supp(q_i(\alpha)) + \LM(f_i) \subseteq \Delta_{F,i} \text{ for all } 1\leq i\leq s,\\
\supp(r(\alpha)) \subseteq \overline{\Delta}_F.\end{cases}$
\end{enumerate}

We begin with $\alpha = \bm{0} \in\mathbb{N}^n$ by setting $q_1(\bm{0}) = \dots = q_s(\bm{0}) = r(\bm{0}) = 0$ and $g(\bm{0}) = f$.

Now we proceed as follows. In all cases, we assume from the start that criteria (1${}_\varepsilon$--4${}_\varepsilon$) are satisfied for all $\varepsilon\prec \alpha$, and we leave it to the reader to check that criteria (1${}_\alpha$--4${}_\alpha$) are satisfied at the end.

\textbf{Case 1:} $\alpha$ is the successor of $\beta$, $\LM(g(\beta)) = \beta$, and some $\LM(f_j) \leq_\div \beta$.

As $\LM(g(\beta)) = \beta$, there exists some $t\in \iota(k)^\times$ such that $g(\beta) = t\bm{x}^\beta + \dots$ (where dots denote $\preceq$-greater terms). Let $1\leq i\leq s$ be minimal with the property that $\LM(f_i) \leq_\div \beta$, and choose $\gamma = \beta - \LM(f_i) \in\mathbb{N}^n$, so that $\LM(\bm{x}^\gamma f_i) = \beta$ by Theorem \ref{thm: LM compatibility}, say $\bm{x}^\gamma f_i = u \bm{x}^\beta + \dots$ for some $u\in\iota(k)^\times$.

In this case, set $q_i(\alpha) = q_i(\beta) + tu^{-1}\bm{x}^\gamma$ and $g(\alpha) = g(\beta) - tu^{-1}\bm{x}^\gamma f_i$. All other elements remain unchanged, i.e. $q_j(\alpha) = q_j(\beta)$ for all $1\leq j\leq s$ with $j\neq i$, and $r(\alpha) = r(\beta)$.

\textbf{Case 2:} $\alpha$ is the successor of $\beta$, and $\LM(g(\beta)) = \beta$, but no $\LM(f_j) \leq_\div \beta$. In this case, suppose again that $g(\beta) = t\bm{x}^\beta + \dots$, and set $g(\alpha) = g(\beta) - t\bm{x}^\beta$ and $r(\alpha) = r(\beta) + t\bm{x}^\beta$. Leave all other elements unchanged, i.e. $q_j(\alpha) = q_j(\beta)$ for all $1\leq j\leq s$.

\textbf{Case 3:} $\alpha$ is the successor of $\beta$, but $\LM(g(\beta)) \neq \beta$. In this case, leave all elements unchanged, i.e. set $g(\alpha) = g(\beta)$, $r(\alpha) = r(\beta)$ and $q_j(\alpha) = q_j(\beta)$ for all $1\leq j\leq s$.

\textbf{Case 4:} $\alpha$ is a limit ordinal. In this case, let $T = \{\gamma\in \mathbb{N}^n : \gamma \prec \alpha\}$. By repeated application of Proposition \ref{propn: convergent subsequences in lex ordering}, we can find an infinite subset $T'\subseteq T$ such that $(g_\beta)_{\beta\in T'}$, $(q_j(\beta))_{\beta\in T'}$ and $(r_\beta)_{\beta\in T'}$ all converge in $R$, and $\sup(T') = \alpha$. Define  $g(\alpha)$, $q_j(\alpha)$ and $r(\alpha)$ to be the limits of these subsequences.

After this inductive process, we will have defined elements $g(\alpha), q_1(\alpha), \dots, q_s(\alpha), r(\alpha)\in R$ satisfying criteria (1${}_\alpha$--4${}_\alpha$) for all $\alpha\in\mathbb{N}^n$.

At this stage, with a final application of the argument in Case 4, but with $T = \mathbb{N}^n$ (and so ``$\sup(T) = \infty$"), we can define $g, q_1, \dots, q_s, r$ to be equal to the respective limits of appropriate subsequences of $g(\alpha)$, etc. This gives the expression $f = g + q_1 f_1 + \dots + q_s f_s$: now we need only notice that the conjunction of the criteria (2${}_\alpha$) for all $\alpha\in\mathbb{N}^n$ implies that $g = 0$, and that the construction of the $q_i$ implies that $\LM(f) \preceq \LM(q_i f_i)$ for all $i$.

Finally, we show that this expression is unique. Suppose that we have two such expressions
$$f = q_1 f_1 + \dots + q_s f_s + r = q'_1 f_1 + \dots + q'_s f_s + r',$$
where $\supp(q_i) + \LM(f_i) \subseteq \Delta_{F,i}$ and $\supp(q'_i) + \LM(f_i) \subseteq \Delta_{F,i}$ for all $1\leq i\leq s$, and $\supp(r) \subseteq \overline{\Delta}_F$ and $\supp(r') \subseteq \overline{\Delta}_F$. This implies that
$$(q_1 - q'_1)f_1 + \dots + (q_s - q'_s)f_s + (r - r') = 0.$$
In case \textbf{(A)}, we can now simply compare the coefficients of each monomial.

In case \textbf{(B)}, note that $\bigcup_{i=1}^t \Delta_{F,i} + (\mathbb{N}\times \{0\}^{n-1}) = \bigcup_{i=1}^t \Delta_{F,i}$ for all $1\leq t\leq s$, and so Lemma \ref{lem: support of sum of monomials} tells us that $\supp(q_i - q'_i) + \LM(f_i) \subseteq \bigcup_{i=1}^t \Delta_{F,i}$ for all $1\leq i\leq t$.

So, if $(q_1-q'_1)f_1 + \dots + (q_s-q'_s)f_s$ is nonzero, its least monomial lies in $\bigcup_{i=1}^s \Delta_{F,i}$, and hence $\LM(r-r') \in \bigcup_{i=1}^s \Delta_{F,i}$. But $\bigcup_{i=1}^s \Delta_{F,i} = \mathbb{N}^n \setminus \overline{\Delta}_F$ by definition, and so it follows that $\LM(r-r') \not\in \supp(r) \cup \supp(r') \subseteq \overline{\Delta}_F$. Now Lemma \ref{lem: checking equality on support} implies that $r = r'$.

Now set $t = s-1$. If $(q_1-q'_1)f_1 + \dots + (q_{s-1}-q'_{s-1})f_{s-1}$ is nonzero, its least monomial lies in $\bigcup_{i=1}^{s-1} \Delta_{F,i}$, and hence $\LM(q_sf_s-q'_sf_s) \in \bigcup_{i=1}^{s-1} \Delta_{F,i}$. Similar to the above, we can deduce from here that $\LM(q_s-q'_s) \not\in \supp(q_s) \cup \supp(q'_s) \subseteq \Delta_{F,s} - \LM(f_s)$, and hence $q_s = q'_s$, and so on.
\end{proof}

We will use the phrase \emph{right-dividing $f$ by $F$} to refer to the process of finding $q_1, \dots, q_s, r$. In this process, $(q_1, \dots, q_s)$ will be called the \emph{tuple of quotients} and $r$ the \emph{remainder}. If $F = (f_1, \dots, f_s)$, we will write $\overline{f}^F = r$ as shorthand for ``$r$ is the remainder when $f$ is right-divided by $F$".

We now derive appropriate analogues of the well-known commutative theory.

\begin{cor}[Left ideal membership]\label{cor: left ideal membership}
Suppose $I$ is a left ideal of $R$ with Gr\"obner basis $G$, and take arbitrary $f\in R$. Then $f\in I$ if and only if $\overline{f}^G = 0$.
\end{cor}

\begin{proof}
If $\overline{f}^G = 0$, then by construction $f\in RG \subseteq I$. Conversely, suppose that $\overline{f}^G = r \neq 0$: then since $G\subseteq I$ we have $r\in I$ by construction. But $\LM(r) \in \supp(r) \subseteq \overline{\Delta}_G$, where $\overline{\Delta}_G = \mathbb{N}^n \setminus \langle \LM(G)\rangle = \mathbb{N}^n \setminus \LM(I)$, so we cannot have $r\in I$.
\end{proof}

Or, rephrasing this:

\begin{cor}[Equality condition for nested ideals]\label{cor: equality of nested ideals}
Suppose $I$ and $J$ are left ideals of $R$ satisfying $I \subseteq J$, and $\LM(I) = \LM(J)$. Then $I = J$.
\end{cor}

\begin{proof}
Let $G$ be a Gr\"obner basis for $I$, and take arbitrary $f\in J$. Then $\overline{f}^G \in J$. But $\supp(\overline{f}^G) \subseteq \overline{\Delta}_G = \mathbb{N}^n \setminus \LM(J)$, so we must have $\overline{f}^G = 0$ and hence $f\in I$ by Corollary \ref{cor: left ideal membership}.
\end{proof}

We also note the following multivariate analogue of the Weierstrass preparation theorem.

\begin{cor}[Multivariate Weierstrass preparation theorem]\label{cor: weierstrass}
Let $0\neq f\in R$. Then there are unique elements $u\in R^\times$ and $F\in R$ satisfying
\begin{enumerate}[label=(\roman*)]
\item $f = uF$ (and so $\LM(f) = \LM(F)$),
\item $\LC(F) = 1$,
\item $\supp(F) \cap (\LM(f) + \mathbb{N}^n) = \{\LM(f)\}$.
\end{enumerate}
\end{cor}

\begin{proof}
Right-dividing $f - \LT(f)$ by $f$ gives us an expression $f - \LT(f) = qf + r$, where $\supp(r) \cap (\LM(f) + \mathbb{N}^n) = \varnothing$. Rearranging this gives $\LT(f) + r = (1-q)f$, and so taking $F = \LC(f)^{-1}(\LT(f) + r)$ and $u = \LC(f)^{-1}(1-q)$ gives the desired result.
\end{proof}

\subsection{Syzygies}

Under our assumptions on $R$ and $\preceq$, we know in particular that Lemma \ref{lem: LTs and LCs} and Theorem \ref{thm: LM compatibility} apply, and we will use these results throughout.

\begin{defn}\label{defn: S-elements}
Given $\alpha,\beta\in\mathbb{N}^n$, their \emph{join} is $\alpha\vee \beta = (\max\{\alpha_1, \beta_1\}, \dots, \max\{\alpha_s, \beta_s\})$. (When $R$ is commutative, $\bm{x}^{\alpha\vee\beta}$ is the least common multiple of $\bm{x}^\alpha$ and $\bm{x}^\beta$.)

Let $g,g'\in R$ be nonzero, with $\LM(g) = \gamma$, $\LM(g') = \gamma'$, and let $\beta, \beta'$ be the unique (div-)minimal exponents such that $\beta + \gamma = \beta' + \gamma'$: explicitly, $\beta = (\gamma\vee \gamma') - \gamma$ and $\beta' = (\gamma\vee \gamma') - \gamma'$. We define their \emph{S-element} to be
$$S(g,g') = \frac{\bm{x}^{\beta} g}{\LC(\bm{x}^{\beta} g)}  - \frac{\bm{x}^{\beta'} g'}{\LC(\bm{x}^{\beta'} g')}.$$
\end{defn}

(As $\LC(r)$ is central and invertible for all nonzero $r\in R$, this fraction notation should cause no confusion, and will significantly improve readability in the upcoming calculations.)

These are a straightforward analogue of S-polynomials in the commutative context: notice that
$$\LT\left(\frac{\bm{x}^{\beta} g}{\LC(\bm{x}^{\beta} g)} \right) = \LT\left(\frac{\bm{x}^{\beta'} g}{\LC(\bm{x}^{\beta'} g')} \right) = \bm{x}^{\gamma \vee \gamma'},$$
and so these terms cancel and $\LM(S(g,g')) \succ \gamma \vee \gamma'$.

\begin{lem}\label{lem: quotient of LCs}
Suppose $\LM(f_i) = \varphi_i$ and $\LM(g_i) = \gamma_i$ for $i = 1, 2$, and suppose that $\varphi_1 + \gamma_1 = \varphi_2 + \gamma_2 = \alpha$, say. Write $\beta = \gamma_1 \vee \gamma_2$ and $\varepsilon = \alpha - \beta$. Then
$$\dfrac{\LC(f_1 g_1)}{\LC(f_2g_2)} = \dfrac{\LC(f_1) \LC(g_1)}{\LC(f_2) \LC(g_2)} \cdot \dfrac{\LC(\bm{x}^{\beta-\gamma_1} \bm{x}^{\gamma_1}) \LC(\bm{x}^{\varepsilon} \bm{x}^{\beta - \gamma_2})}{\LC(\bm{x}^{\beta-\gamma_2} \bm{x}^{\gamma_2}) \LC(\bm{x}^{\varepsilon} \bm{x}^{\beta - \gamma_1})}.$$ 
\end{lem}

\begin{proof}
Using Lemma \ref{lem: LTs and LCs}(iii), for each $i$, we calculate $\LC(\bm{x}^{\varepsilon} \bm{x}^{\beta - \gamma_i} \bm{x}^{\gamma_i})$ in two ways:
$$\LC(\bm{x}^{\varepsilon} \bm{x}^{\beta - \gamma_i} \bm{x}^{\gamma_i}) = \LC(\bm{x}^{\varepsilon+\beta-\gamma_i} \bm{x}^{\gamma_i}) \LC(\bm{x}^{\varepsilon} \bm{x}^{\beta - \gamma_i}) = \LC(\bm{x}^{\varepsilon} \bm{x}^{\beta}) \LC(\bm{x}^{\beta - \gamma_i} \bm{x}^{\gamma_i}).$$
But $\varepsilon+\beta-\gamma_i = \varphi_i$, and hence
$$\LC(\bm{x}^{\varphi_i} \bm{x}^{\gamma_i}) = \dfrac{\LC(\bm{x}^{\varepsilon} \bm{x}^{\beta}) \LC(\bm{x}^{\beta - \gamma_i} \bm{x}^{\gamma_i})}{\LC(\bm{x}^{\varepsilon} \bm{x}^{\beta - \gamma_i})}.$$
Now combine this with the fact that $\LC(f_ig_i) = \LC(f_i)\LC(g_i)\LC(\bm{x}^{\varphi_i} \bm{x}^{\gamma_i})$, which follows from Lemma \ref{lem: LTs and LCs}(i).
\end{proof}

\begin{lem}\label{lem: cancelling LTs using S-elements}
Suppose that $\LM(f_ig_i) = \alpha$ for $i = 1, \dots, s$, but that $\LM(f_1 g_1 + \dots + f_s g_s) \succ \alpha$. Write $\LM(f_i) = \varphi_i$ and $\LM(g_i) = \gamma_i$ for all $i$, so that $\varphi_i + \gamma_i = \alpha$ for all $i$. Define $\beta = \gamma_1 \vee \gamma_2$ and $\varepsilon = \alpha - \beta$, and then set
$$f'_1 = f_1 - \LC(f_1) \dfrac{\bm{x}^\varepsilon \bm{x}^{\beta - \gamma_1}}{\LC(\bm{x}^\varepsilon \bm{x}^{\beta - \gamma_1})}, \qquad f'_2 = \dfrac{\LC(f_1g_1)}{\LC(f_2g_2)}\cdot  \LC(f_2) \dfrac{\bm{x}^\varepsilon \bm{x}^{\beta-\gamma_2}}{\LC(\bm{x}^\varepsilon \bm{x}^{\beta-\gamma_2})} + f_2,$$
and $S_{12} = \LC(f_1) \LC(g_1) \bm{x}^\varepsilon S(g_1,g_2)$. Then
\begin{enumerate}[label=(\roman*)]
\item $f_1g_1 + f_2g_2 = f'_1 g_1 + f'_2 g_2 + S_{12}$,
\item $\LM(f'_1 g_1) \succ \alpha$ and $\LM(S_{12}) \succ \alpha$.
\end{enumerate}
\end{lem}

\begin{proof}
Tedious calculation using Lemma \ref{lem: quotient of LCs}.
\end{proof}

\subsection{Buchberger's criterion}

Let $I$ be a nonzero left ideal of $R$, and let $G = (g_1, \dots, g_s)$ be a fixed left ideal basis for $I$. We ask: how can we tell whether $G$ is a Gr\"obner basis for $I$? In this subsection, we develop a form of Buchberger's criterion in this noncommutative context.

\begin{thm}[Buchberger's criterion]\label{thm: S-elements determine Groebner bases}
Suppose $I$ is a nonzero left ideal of $R$ and $G = (g_1, \dots, g_s)$ is a left ideal basis for $I$. Then $G$ is a Gr\"obner basis for $I$ if and only if $\overline{S(g_i,g_j)}^G = 0$ for all $1\leq i, j\leq s$.
\end{thm}

\begin{rk}\label{rk: representations of f}
Let $f\in I$ be any nonzero element, and suppose that $G = (g_1, \dots, g_s)$ is a left ideal basis for $I$. By definition, this means that $f$ can be written as $f_1 g_1 + \dots + f_s g_s$ for some choice of elements $F = (f_1, \dots, f_s)\in R^s$. However, this $F$ may not be unique.

We temporarily write $\mathcal{X}_f = \{F = (f_1, \dots, f_s) \in R^s : f = f_1 g_1 + \dots + f_s g_s\}$. By Proposition \ref{propn: LM is filtered}, we always have $\LM(f) \succeq \min_i\{\LM(f_ig_i)\}$ for any choice of $F = (f_1, \dots, f_s) \in \mathcal{X}_f$. If it happens that $\LM(f) = \min_i\{\LM(f_ig_i)\}$ for some $F = (f_1, \dots, f_s)\in\mathcal{X}_f$, so that $\LM(f) = \LM(f_ig_i)$ for some $i$, then (by definition of a monomial order) we have $\LM(f) = \LM(f_i) + \LM(g_i) \in \langle \LM(g_i) \rangle$, so $\LM(f) \in \langle \LM(G)\rangle$.  However, if $\LM(f) \neq \min_i\{\LM(f_ig_i)\}$, we cannot yet conclude this. However, we will show that this can always be achieved by choosing an appropriate $F$.

For ease of notation, if $F = (f_1, \dots, f_s)\in \mathcal{X}_f$, we will write $m(F) = \min_i\{\LM(f_ig_i)\}$, and we will set $m(\mathcal{X}_f) = \{m(F) : F\in\mathcal{X}_f\}$.
\end{rk}

\begin{lem}\label{lem: there exists an LM-maximal representative of f}
There exists a $\preceq$-maximal element $\alpha$ of $m(\mathcal{X}_f)$.
\end{lem}

\begin{proof}
As $f$ is nonzero, we always have $m(F) \preceq \LM(f)$ by Proposition \ref{propn: LM is filtered}, so $m(\mathcal{X}_f)$ is bounded above. Let $\alpha = \sup(m(\mathcal{X}_f))$: if $\alpha\in m(\mathcal{X}_f)$, then there is nothing to prove. So we will assume for contradiction that $\alpha\not\in m(\mathcal{X}_f)$.

Set $T = m(\mathcal{X}_f)$, and for each $\beta\in T$, choose $F_\beta = (f_{1\beta}, \dots, f_{s\beta}) \in \mathcal{X}_f$ such that $m(F_\beta) = \beta$. By applying Proposition \ref{propn: convergent subsequences in lex ordering} $s$ times, once for the sequence $(f_{i\beta})$ in each coordinate $1\leq i\leq s$, we may replace $T$ with a smaller ($\mathbb{N}$-indexed, increasing) set $T'$, still satisfying $\sup(T') = \alpha$, such that the sequence $(f_{i\beta})_{\beta\in T'}$ converges in $R$ for each $1\leq i\leq s$.

In particular, this means that $(F_\beta)_{\beta\in T'}$ converges, say to $F = (f_1, \dots, f_s)\in R^s$. However, since the map $\theta: R^s\to R$ sending $(r_1, \dots, r_s)$ to $r_1 g_1 + \dots + r_s g_s$ is continuous, and $\theta(F_\beta) = f$ for all $\beta\in T'$, we must have $\theta(F) = \lim_{\beta\in T'} \theta(F_\beta) = f$, i.e. $F\in \mathcal{X}_f$.

Finally, as $m(F_\beta) = \beta$ for each $\beta\in T'$ and $T'$ is infinite, we have $\beta \preceq m(F) \preceq \alpha$ for all $\beta\in T'$, and so $\sup(T') \preceq m(F)  \preceq \alpha$: in other words, $m(F) = \alpha$.
\end{proof}

So we now fix $\alpha = \sup(m(\mathcal{X}_f))$ and $F = (f_1, \dots, f_s)\in \mathcal{X}_f$ such that $m(F) = \alpha$. The argument of Remark \ref{rk: representations of f} shows that $\LM(f) \succeq \alpha$. In fact, we will prove:

\begin{lem}\label{lem: LM(f) is maximal}
Suppose that $\overline{S(g_i, g_j)}^G = 0$ for all $1\leq i,j\leq s$. Then $\LM(f) = \alpha$.
\end{lem}

\begin{proof}
Suppose not: then there are at least two indices $i$ such that $\LM(f_i g_i) = \alpha$, but in the sum $f_1g_1 + \dots + f_sg_s$, they all cancel. Applying Lemma \ref{lem: cancelling LTs using S-elements} repeatedly, we get
\begin{align}\label{eqn: cancelling common least terms}
f = f_1'g_1 + \dots + f_s'g_s + S'_1 + \dots + S'_\ell,
\end{align}
where each $\LM(f'_ig_i) \succ \alpha$, each $\LM(S'_j) \succ \alpha$ and each $S'_j$ is a left multiple of some $S(g_i, g_j)$. Now, the hypothesis $\overline{S(g_i, g_j)}^G = 0$ implies that each $S'_j$ can be written as
\refstepcounter{equation}\label{eqn: each S-element}
\begin{align}
\notag \tag{\ref{eqn: each S-element}.$j$}
S'_j = r_{j,1} g_1 + \dots + r_{j,s} g_s,
\end{align}

where each $\LM(r_{j,m} g_m) \succeq \LM(S'_j) \succ \alpha$. Now, substituting (\ref{eqn: each S-element}.$j$) into (\ref{eqn: cancelling common least terms}) for all $1\leq j\leq \ell$, we get an expression
$$f = f_1''g_1 + \dots + f_s''g_s$$
where each $\LM(f_i''g_i)\succ \alpha$. But now the tuple $F'' = (f''_1, \dots, f''_s)\in\mathcal{X}_f$ satisfies $m(F'') \succ \alpha$, a contradiction to our choice of $\alpha$.
\end{proof}

\textit{Proof of Theorem \ref{thm: S-elements determine Groebner bases}.} One direction follows from Corollary \ref{cor: left ideal membership}: since each $S(g_i,g_j)\in I$, and $G$ is a Gr\"obner basis for $I$, we must have $\overline{S(g_i,g_j)}^G = 0$.

Suppose instead that $\overline{S(g_i,g_j)}^G = 0$ for all $1\leq i,j\leq s$. Take $f\in I$: we need to show that $\LM(f) \in \langle \LM(G)\rangle$. Then $m(\mathcal{X}_f)$ has a maximal element $\alpha$ by Lemma \ref{lem: there exists an LM-maximal representative of f}, and $\LM(f) = \alpha$ by Lemma \ref{lem: LM(f) is maximal}. The result now follows from the argument of Remark \ref{rk: representations of f}.\qed

\subsection{An application: the two-dimensional case}

Let $R = \mathbb{F}_p[[x]][[y; \sigma, \delta]]$, where $(\sigma, \delta)$ is $\mathbb{F}_p$-linear. Throughout we will take $\preceq$ to be the \emph{lexicographic} ordering.

\begin{rks}
We mention special cases of this family of rings that have been studied in the literature:

\begin{enumerate}[label=(\roman*)]
\item The ring $\mathbb{F}_p[[G]]$ as defined and studied in \cite[Example 2.2]{venjakob} is a special case of our $R$, and the result of \cite[Theorem 7.1]{venjakob} is that $\mathbb{F}_p[[G]]$ is a UFD.
\item Suppose $\sigma = \mathrm{id}$ and $\delta(x) = x^m$ for some $m\geq 2$. Then, by induction on $\ell$,
$$\delta^\ell(x^n) = \left(\prod_{i=0}^{\ell-1} (n+i(m-1))\right)x^{n+\ell(m-1)}$$
for all $\ell\geq 0$ and $n \geq 0$. Now, if $m-1$ is assumed nonzero mod $p$, the product $\prod_{i=0}^{p-1} (n+i(m-1))$ is equal to zero mod $p$, and so $\delta^p(x^n) = 0$ for all $n\geq 0$: in particular, $\delta$ is nilpotent, and so the ring $R$ can be studied through the lens of \cite[\S 3]{bergen-grzeszczuk-skew}. However, if $m\equiv 1\bmod p$, then $\delta$ is no longer even locally nilpotent.
\end{enumerate}
\end{rks}

Let $J\subsetneq I$ be two-sided ideals with $I\neq 0$. Choose a Gr\"obner basis $G_J = (h_1, \dots, h_t)$ for $J$ (this can be empty if $J = 0$), and extend it to a Gr\"obner basis $G = (h_1, \dots, h_t, g_1, \dots, g_s)$ for $I$ satisfying $\LM(g_1) \succ \dots \succ \LM(g_s)$. We may assume, without loss of generality, that no $g_i\in J$.

We will not need to know much about $G_J$, and will mostly be focused on the elements $g_1, \dots, g_s$. So, for ease of notation, we begin by relabelling $\Delta_{J,1} = \Delta_{G,1}$, $\dots$, $\Delta_{J,t} = \Delta_{G,t}$, and $\Delta_{I,1} = \Delta_{G,t+1}$, $\dots$, $\Delta_{I,s} = \Delta_{G,t+s}$.

Now set $\LM(g_i) = (\gamma_{i,1}, \gamma_{i,2})$ for each $1\leq i\leq s$. We can explicitly calculate
\begin{align}\label{eqn: 2d Delta calculation}
\begin{cases}
\Delta_{I,i} \subseteq [\gamma_{i,1}, \gamma_{i-1, 1}) \times [\gamma_{i,2}, \infty)& \text{for all }2\leq i\leq s,\\
\Delta_{I,1} \subseteq [\gamma_{1,1}, \infty) \times [\gamma_{1,2}, \infty).
\end{cases}
\end{align}

\begin{multicols}{2}
\centerline{
\begin{tikzpicture}[scale=0.9]
    \draw [->](-0.5,0) -- (7.8,0);
    \draw [->](0,-0.5) -- (0,5.8);
    \node[] at (7.5,-0.3) {$a$};
    \node[] at (-0.3,5.5) {$b$};

	\filldraw [fill=gray!60, draw=gray!60] (1.6,4.5) rectangle (7.8,5.8);
	\filldraw [fill=gray!60, draw=gray!60] (2.7,3.7) rectangle (7.8,5.8);
	\filldraw [fill=gray!60, draw=gray!60] (3.3,2.8) rectangle (7.8,5.8);
	\filldraw [fill=gray!60, draw=gray!60] (4.2,1.3) rectangle (7.8,5.8);
	\filldraw [fill=gray!60, draw=gray!60] (6.5,0.8) rectangle (7.8,5.8);
    \draw [-, draw=black] (1.6,5.8) -- (1.6,4.5) -- (2.7,4.5) -- (2.7,3.7) -- (3.3,3.7) -- (3.3,2.8) -- (4.2,2.8) -- (4.2,1.3) -- (6.5,1.3) -- (6.5,0.8) -- (7.8,0.8);
    
	\filldraw [fill=gray!30, draw=gray!30] (0.4,5.5) rectangle (7.8,5.8);
	\filldraw [fill=gray!30, draw=gray!30] (4.6,4) rectangle (7.8,5.8);
	\filldraw [fill=gray!30, draw=gray!30] (5.5,2.7) rectangle (7.8,5.8);
    \draw [-, draw=black] (0.4,5.8) -- (0.4,5.5) -- (4.6,5.5) -- (4.6,4) -- (5.5,4) -- (5.5,2.7) -- (7.8,2.7);
    
    \node[fill=white] at (4.4,3.4) {$\LM(I)$};
    \node[fill=white] at (6.4,4.9) {$\LM(J)$};
\end{tikzpicture}
}

\centerline{
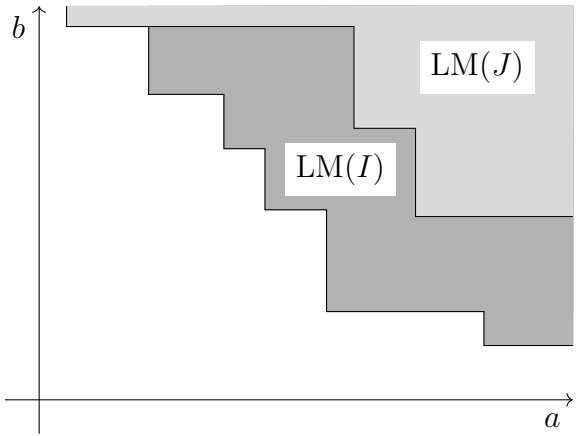
Figure 2: $\LM(J) \subseteq \LM(I)$
}

\centerline{
\begin{tikzpicture}[scale=0.9]
    \draw [->](-0.5,0) -- (7.8,0);
    \draw [->](0,-0.5) -- (0,5.8);
    \node[] at (7.5,-0.3) {$a$};
    \node[] at (-0.3,5.5) {$b$};

    \draw [-, draw=black] (1.6,5.8) -- (1.6,4.5) -- (2.7,4.5) -- (2.7,3.7) -- (3.3,3.7) -- (3.3,2.8) -- (4.2,2.8) -- (4.2,1.3) -- (6.5,1.3) -- (6.5,0.8) -- (7.8,0.8);
    
	\draw [dashed, draw=black] (2.7,5.8) -- (2.7,4.5);
	\draw [dashed, draw=black] (3.3,5.8) -- (3.3,3.7);
	\draw [dashed, draw=black] (4.2,5.8) -- (4.2,2.8);
	\draw [dashed, draw=black] (6.5,5.8) -- (6.5,1.3);
    
    \node[fill=white] at (7.2,1.7) {$\Delta_{I,1}$};
    \node[fill=white] at (5.4,2.1) {$\Delta_{I,2}$};
    \node[] at (2.15,5) {$\Delta_{I,s}$};
    \node[] at (0.9,4.3) {$\LM(g_s)$};
    \node[] at (1.8,3.5) {$\LM(g_{s-1})$};
    \node[] at (2.4,2.38) {$\ddots$};
    \node[] at (3.5,1.1) {$\LM(g_2)$};
    \node[] at (5.8,0.6) {$\LM(g_1)$};

	\filldraw [fill=gray!30, draw=gray!30] (0.4,5.5) rectangle (7.8,5.8);
	\filldraw [fill=gray!30, draw=gray!30] (4.6,4) rectangle (7.8,5.8);
	\filldraw [fill=gray!30, draw=gray!30] (5.5,2.7) rectangle (7.8,5.8);
    \draw [-, draw=black] (0.4,5.8) -- (0.4,5.5) -- (4.6,5.5) -- (4.6,4) -- (5.5,4) -- (5.5,2.7) -- (7.8,2.7);

\end{tikzpicture}
}

\centerline{
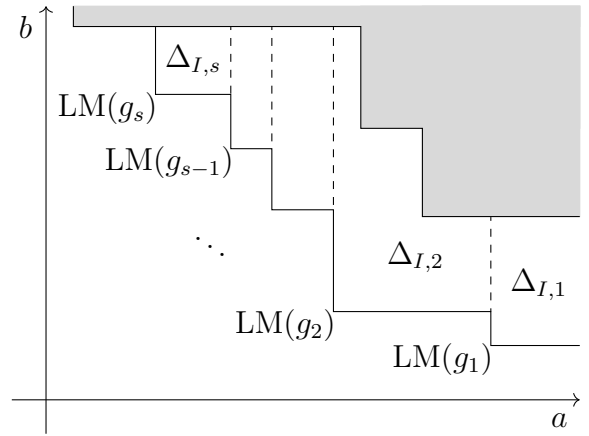
Figure 3: $\Delta_{I,i}$
}
\end{multicols}

The following is now clear:

\begin{lem}\label{lem: 2d lex order on Deltas}
For all $1 \leq i < j \leq s$ and all $\alpha \in \Delta_{I,i}$ and $\beta\in \Delta_{I,j}$, we have $\alpha \succ \beta$.\qed
\end{lem}

Now, for any $0\leq r\leq s$, define $G_r = (h_1, \dots, h_t, g_1, \dots, g_r)$, and set $I_r = RG_r + J$. (For convenience, we are including $I_0 := J$ and $I_s := I$.) Note that, at the moment, we only know that $I_r$ is a \emph{left} ideal and that $G_r$ is a \emph{left ideal} basis for $I_r$.

\begin{lem}\label{lem: 2d case, condition for f to be in smaller ideal}
Take $f\in I\setminus J$. Then $f\in I_r$ if and only if there exists $j\in J$ such that $\LM(f-j) \succeq \LM(g_r)$.
\end{lem}

\begin{proof}
Suppose first that $f\in I_r$, so that $f = q'_1 h_1 + \dots + q'_t h_t + q_1 g_1 + \dots + q_r g_r$ for some $q_i, q'_i$. Set $j = q'_1 h_1 + \dots + q'_t h_t\in J$. Then $\LM(f-j) \succeq \min_{1\leq i\leq r}\{\LM(q_ig_i)\}$ by Proposition \ref{propn: LM is filtered}, and so in particular, $\LM(f-j) \succeq \LM(q_ig_i)$ for some $i$. But then $\LM(q_ig_i) \succeq \LM(g_i)$ by Theorem \ref{thm: LM compatibility}, and $\LM(g_i) \succeq \LM(g_r)$ for all $1\leq i\leq r$ by our choice of ordering.

On the other hand, for any $0\neq f\in I$, we can use Theorem \ref{thm: division theorem} to right-divide $f$ by $G$ and obtain an expression of the form $f = q'_1 h_1 + \dots + q'_t h_t + q_1 g_1 + \dots + q_s g_s$, where $\supp(q_i) + \LM(g_i) \subseteq \Delta_{I,i}$ for all $1\leq i\leq s$. Again, set $j = q'_1 h_1 + \dots + q'_t h_t\in J$. If any of $q_{r+1}, \dots, q_s$ is nonzero in this expression, then fix $u$ to be maximal such that $q_u\neq 0$. Then in particular $\LM(q_u g_u) \in \Delta_{I,u}$, and Lemma \ref{lem: 2d lex order on Deltas} implies that $\LM(q_u g_u) \prec \LM(q_i g_i)$ for all $1\leq i\leq u-1$, and hence $\LM(f-j) = \LM(q_u g_u) \in \Delta_{I,u}$.

Now, if $j'\in J$ is arbitrary, then $\LM(j-j') \in \LM(J)$, and in particular we can never have $\LM(j-j') = \LM(f-j)$, as $\LM(J) \cap \Delta_{I,u} = \varnothing$. Proposition \ref{propn: LM is filtered}(ii) now implies that $\LM(f-j') = \min_\preceq\{\LM(f-j), \LM(j-j')\}$, and in particular that $\LM(f-j') \preceq \LM(f-j) \in \Delta_{I,u}$ for all $j'\in J$. But if we suppose additionally that $\LM(f-j') \succeq \LM(g_r)$ for some $j'\in J$, this leads to a contradiction: so we must have $q_{r+1} = \dots = q_s = 0$, and so $f\in I_r$.
\end{proof}

\begin{lem}\label{lem: smaller ideal is two-sided with the right groebner basis}
$I_r$ is a two-sided ideal, and $G_r$ is a Gr\"obner basis for $I_r$.
\end{lem}

\begin{proof}
By assumption, $I_r$ is a left ideal.

Take $f\in I_r$ and $a\in R$: we will show that $fa\in I_r$, and hence that $I_r$ is a two-sided ideal. If $f\in J$, then $fa\in J\subseteq I_r$, so suppose not. Then we may apply Lemma \ref{lem: 2d case, condition for f to be in smaller ideal} to see that there exists $j\in J$ satisfying $\LM(f-j) \succeq \LM(g_r)$. On the other hand, $0\neq fa\in I$, and $\LM((f-j)a) \geq_\div \LM(f-j) \succeq \LM(g_r)$, so a second application of Lemma \ref{lem: 2d case, condition for f to be in smaller ideal} shows that $fa\in I_r$. So $I_r$ is a two-sided ideal.

To show that $G_r$ is a Gr\"obner basis, we will apply Theorem \ref{thm: S-elements determine Groebner bases}. We know that each $S(h_i, h_j)\in J$, so $\overline{S(h_i, h_j)}^{G_J} = 0$ and hence $\overline{S(h_i, h_j)}^{G_r} = 0$. Now take $f = S(h_i, g_j)\in I_r$ (for $1\leq i\leq t$ and $1\leq j\leq r$) or $f = S(g_i, g_j) \in I_r$ (for $1\leq i,j\leq r$), and apply Theorem \ref{thm: division theorem} to right-divide $f$ by $G$, getting an expression of the form
\begin{align}\label{eqn: expansion of f}
f = q'_1 h_1 + \dots + q'_t h_t + q_1 g_1 + \dots + q_s g_s,
\end{align}
 where $\supp(q'_j) + \LM(h_j) \in \Delta_{J,j}$ and $\supp(q_i) + \LM(g_i) \in \Delta_{I,i}$. Now arguing as in the proof of Lemma \ref{lem: 2d case, condition for f to be in smaller ideal} will show that $q_{r+1} = \dots = q_s = 0$, and by the uniqueness in Theorem \ref{thm: division theorem}, equation (\ref{eqn: expansion of f}) must now be the representation of $f$ afforded by right-division by $G_r$, and in particular the remainder must be $0$.
\end{proof}

\begin{cor}\label{cor: 2d rings are polynormal}
$R$ is polynormal.
\end{cor}

\begin{proof}
The normal element separating $I$ and $J$ is $g_1 + J \in I/J$. To see this, let $r\in R$ be arbitrary: then Lemma \ref{lem: smaller ideal is two-sided with the right groebner basis} implies that $I_1$ is a two-sided ideal, so $g_1r\in I_1$. The same lemma says that $G_1 = (h_1, \dots, h_t, g_1)$ is a Gr\"obner basis for $I_1$, so by Corollary \ref{cor: left ideal membership}, we must have $\overline{g_1 r}^{G_1} = 0$. Explicitly, there exist $q_1, \dots, q_t$ and $r'$ in $R$ such that $g_1 r = q_1 h_1 + \dots + q_t h_t + r' g_1$, or in other words $g_1 r - r' g_1 \in J$.
\end{proof}

We end this section by asking when the methods of this section will work in greater generality. More precisely:

\begin{qn}\label{qn}
Let $J\subseteq I$ be ideals of a ring $R = A[[\bm{x}; \bm{\sigma}, \bm{\delta}]]^n$, with $I\neq 0$. Suppose $J$ has fixed Gr\"obner basis $G_J = (h_1, \dots, h_t)$ and $I$ has fixed Gr\"obner basis $G = (h_1, \dots, h_t, g_1, \dots, g_s)$, where no $g_i\in J$. Define the \emph{truncated} bases $G_r = (h_1, \dots, h_t, g_1, \dots, g_r)$ for all $1\leq r\leq s$, and the left ideals $I_r = RG_r$. For which $R$, $I$ and $G$ do we have that (i) $I_r$ is a \emph{two-sided} ideal, and (ii) $G_r$ is a (left) Gr\"obner basis for $I_r$, for all $1\leq r\leq s$?
\end{qn}

\begin{rks}
$ $

\begin{enumerate}
\item We have not included $G_J$ in the above question, as the details of $G_J$ played very little role in the methods of this section.
\item It is sufficient for $G$ to satisfy an analogue of Lemma \ref{lem: 2d lex order on Deltas}. However, this does not hold in general when $R$ has dimension at least three, even if $R$ is commutative.
\item Let $R = \mathbb{F}_p[[x,y,z]]$ be the \emph{commutative} power series ring, and let $G = (y^2z+xz^2, yz^2, z^3)$. Set $J = 0$ and $I = RG$. Then $G$ is a (minimal) Gr\"obner basis for $I$ with $3$ elements, so in the notation of the question, $G = G_3$. But the truncated basis $G_2$ is \emph{not} a Gr\"obner basis for $I_2$, as $S(y^2z+xz^2, yz^2)$ has nonzero remainder on division by $G_2$, so $G$ does not satisfy the properties of the question. However, the (non-minimal!) Gr\"obner basis $G' = (xz^3, y^2z+xz^2, yz^2, z^3)$ for $I$ \emph{does} satisfy the properties of the question.
\end{enumerate}
\end{rks}

\bibliography{../../biblio/biblio-all}
\bibliographystyle{plain}

\end{document}